\documentclass[11pt,reqno]{amsart}
\usepackage{amssymb,amsmath,color,slashbox}
\usepackage{tikz-dependency}
\usepackage{tikz}
\usepackage{enumerate}
\usepackage{hyperref}

\allowdisplaybreaks

\newtheorem{theorem}{Theorem}[section]
\newtheorem{proposition}[theorem]{Proposition}
\newtheorem{lemma}[theorem]{Lemma}
\newtheorem{definition}[theorem]{Definition}

\newtheorem{corollary}[theorem]{Corollary}
\newtheorem{example}[theorem]{Example}
\newtheorem{remark}[theorem]{Remark}
\newtheorem{conjecture}[theorem]{Conjecture}
\newtheorem{question}[theorem]{Question}

\numberwithin{equation}{section}

\newcommand{\SSSS}{{\mathfrak S}}

\newcommand{\hz}{{\hat 0}}
\newcommand{\ho}{{\hat 1}}

\newcommand{\Nnn}{{\mathbb N}}

\newcommand{\Rrr}{{\mathbb R}}
\newcommand{\Zzz}{{\mathbb Z}}

\newcommand{\half}[1]{#1/2}
\newcommand{\binomial}[2]{\binom{#1}{#2}}

\DeclareMathOperator{\Asc}{Asc}
\DeclareMathOperator{\asc}{asc}

\DeclareMathOperator{\Des}{Des}
\DeclareMathOperator{\des}{des}

\DeclareMathOperator{\rank}{rank}

\newcommand{\Ass}{\mathcal{A}}
\newcommand{\Cyc}{\mathcal{C}}
\newcommand{\Perm}{\mathcal{P}}

\begin{document}

\title[The toric $g$-vector of nestohedra]
{The toric g-vector of nestohedra}

\author{Richard EHRENBORG, G\'abor HETYEI and Margaret READDY}

\address{Department of Mathematics, University of Kentucky, Lexington,
  KY 40506-0027.\hfill\break \tt http://www.math.uky.edu/\~{}jrge/,
  richard.ehrenborg@uky.edu.}

\address{Department of Mathematics and Statistics,
  UNC-Charlotte, Charlotte NC 28223-0001.\hfill\break
\tt http://www.math.charlotte.edu/ghetyei,
ghetyei@charlotte.edu.}

\address{Department of Mathematics, University of Kentucky, Lexington,
  KY 40506-0027.\hfill\break
\tt http://www.math.uky.edu/\~{}readdy/,
margaret.readdy@uky.edu.}

 % Lattice polytopes in convex geometry (including relations with commutative algebra and algebraic geometry


\date{\today}

\begin{abstract}
We show the entries of the toric $g$-vector of any dual simplicial polytope 
is a 
nonnegative integer linear combination of its $\gamma$-vector entries.
We give combinatorial interpretations of the toric $g$-vector of 
three classical polytopes: 
the associahedron, the cyclohedron and the permutahedron.
Each is expressed in terms of 
$123$-avoidance on a set of structures,
namely,
parking functions,
functions and  
parking trees.
Using
our  notion of parking trees,
we extend our combinatorial model to all chordal nestohedra.
As a corollary we obtain 
combinatorial proofs of nonnegativity of the
toric $g$-vector for all of these polytopes.
\end{abstract}

\maketitle

\section{Introduction}

In his 1987 paper on intersection cohomology and generalized
$h$-vectors,
Richard Stanley computed the
toric $g$-vector of the 
permutahedron up to dimension $5$ 
as a potential fruitful example for understanding this invariant;
see~\cite[Page~195]{Stanley_toric_g}.
Twenty-five years later in the second edition of his book,
{\em Enumerative Combinatorics,}
Stanley described the
toric $g$-vector of an Eulerian poset
as ``an exceedingly subtle invariant'';
see~\cite[Section~3.16]{Stanley_EC1}.
For the face lattice of rational
convex polytopes, 
Stanley
proved that 
the nonnegativity of the 
toric $g$-polynomial coefficients 
follows from the hard Lefschetz theorem
for the intersection homology of projective toric varieties~\cite[Corollary 3.2]{Stanley_toric_g}.
For general convex polytopes,
the nonnegativity of the toric $g$-polynomial
follows from Karu's
hard Lefschetz theorem for combinatorial intersection homology~\cite[Theorem 0.1]{Karu}.

In the present paper we not only
give a closed form formula for the toric $g$-vector of 
the permutahedron and 
of several other simple polytopes, but we also find a combinatorial
interpretation in each case.
The key result making such computations and
interpretations possible is Theorem~\ref{theorem_peaks}.
It expresses the entries of the toric $g$-vector of a simple polytope as
nonnegative integer linear combinations of its $\gamma$-vector entries,
an invariant of flag homology spheres, more generally,
Eulerian simplicial complexes introduced by Gal~\cite{Gal}.
This result naturally inspires us to revisit combinatorial models for the
$\gamma$-vectors of classical simple polytopes.

The $n$-dimensional cube is the most basic example
of a simple polytope
where the $\gamma$-vector 
is straightforward and the toric $g$-vector is well known;
see Corollaries~\ref{corollary_nonsingleton}
and~\ref{corollary_Denise--Simion}.
Our approach offers new enumerative results.
We give a combinatorial interpretation
of the toric $g$-vector of the associahedron
in Theorem~\ref{theorem_Margie}:  the $k$th entry
is given by the number of $123$-avoiding parking functions 
$f : [n] \longrightarrow [n]$ that have exactly $k$ ascents. 
The $\gamma$-vector of the associahedron
is discussed alongside that of the permutahedron and of the cyclohedron
in the work of Postnikov, Reiner and
Williams~\cite{Postnikov_Reiner_Williams}.
As in their results on the $\gamma$-vector,
our computation of the toric $g$-vector of the
cyclohedron is analogous to that of the associahedron.
For the cyclohedron we replace
the set of $123$-avoiding parking functions 
with the set of $123$-avoiding functions $f : [n] \longrightarrow [n]$. 

Postnikov, Reiner and
Williams~\cite{Postnikov_Reiner_Williams} expressed the $\gamma$-vector
of the permutahedron in terms of the peak statistics of all permutations
containing no double descents and no final descent.
Our key observation is that their result may
be translated into a result on the peak statistics of all plane
$0$-$1$-$2$ trees having a descending labeling on the vertices. 
The equivalence of the two statistics may be shown using a restricted
variant of the Foata--Strehl group action, first considered by
Br\"and\'en~\cite{Branden}. Having a descending vertex labeling is also
part of our definition of a {\em parking tree} that we introduce as a
simplified variant of the definitions given in work of
Yan~\cite{Yan_chapter} and inspired by an exercise of
Sagan~\cite[Ch~1.\ Exercise~(32)(c)]{Sagan}.   
In 
Theorem~\ref{theorem_Margie3} we express the toric $g$-vector of
the permutahedron in terms of the peak statistics of parking trees
representing $123$-avoiding parking functions. This approach 
extends to all {\em chordal nestohedra}, studied by Postnikov, Reiner and
Williams~\cite{Postnikov_Reiner_Williams}.

Our paper raises many open questions, the most important is the
following. It is inspired by Theorem~\ref{theorem_peaks} and the
Nevo--Petersen conjecture~\cite[Conjecture 1.4]{Nevo_Petersen} stating
that the $\gamma$-vector entries of a simplicial complex that is a flag
homology sphere satisfy the Kruskal--Katona inequalities.
\begin{conjecture}
\label{conjecture_simplicial}  
When the $\gamma$-vector entries of a simple polytope satisfy the
Kruskal--Katona inequalities, the same holds for its toric $g$-vector. 
\end{conjecture}

The paper is organized as follows.
In the preliminaries section we introduce two essential bijections.
The first, due to Krattenthaler, is between $123$-avoiding permutations
and Dyck paths. The second, due to Garsia and Haiman,
is between functions on the set $[n] = \{1, 2, \ldots, n\}$ and pairs of
compatible permutations and lattice paths.
Their important observation is that when the function
is a parking function, the associated lattice path is a Dyck path.
We also discuss the Foata--Strehl group action on
the symmetric group.

In Section~\ref{section_the_toric_g-polynomial}
we show how to compute the toric $g$-vector of a simple polytope
in terms of its $\gamma$-vector. 
This computation involves
the toric $g$-contribution polynomials $g_{n,j}(x)$
whose coefficients are described in terms of counting
peaks in Dyck paths.
In Section~\ref{section_compatible} we give two
interpretations of the coefficients of these polynomials.
One of these is in terms of ascent sets
in $123$-avoiding permutations.

In Section~\ref{section_Generating_functions}
we obtain generating functions for the
toric $g$-contribution polynomials.
We combine the bijections of Garsia--Haiman
and Krattenthaler
in Section~\ref{section_associahedron}
to prove our combinatorial
interpretation of the toric $g$-vector of the associahedron.
We modify our arguments 
to obtain a similar result for the cyclohedron
in Section~\ref{section_cyclohedron}.

Parking trees are introduced in Section~\ref{section_parking_trees}.
They form the essential structure in order to
obtain an explicit expression for
the toric $g$-polynomial of the permutahedron,
completing Stanley's quest to understand
these polynomials;
see Section~\ref{section_permutahedron}.
Finally in Section~\ref{section_chordal_nestohedra}
we extend our results to
nestohedra associated with
chordal building sets.
We end with open questions and comments
in Section~\ref{section_Concluding_remarks}.

\section{Preliminaries}
\label{section_preliminaries}

\subsection{The toric $g$-vector}
\label{subsection_The_toric_g-vector}
Following 
Stanley~\cite[Section 2]{Stanley_toric_g}
and~\cite[Section 3.16]{Stanley_EC1},
for an Eulerian poset $P$ with rank function~$\rho$,
the {\em toric $h$-} and {\em toric $g$-polynomials},
denoted
$f(P,x)$ and $g(P,x)$,
are recursively defined
as follows.

\begin{enumerate}
\item
$f({\bf 1},x) = g({\bf 1},x) = 1$,
where
${\bf 1}$ denotes the poset consisting of a single element.

\item
If $\rank P = n+1 > 0$ then 
the toric $h$-polynomial $f(P,x)$ has
degree $n$.
Writing 
$$
     f(P,x) = h_0 + h_1 x + \cdots + h_n x^n,
$$
define
$$
  g(P,x) = h_0 + (h_1 - h_0)x + (h_2 - h_1)x^2 + \cdots + 
           (h_m - h_{m-1})x^m,
$$
where
$m = \lfloor n/2 \rfloor$.

\item
If $\rank P = n+1 > 0$ then define
$$
     f(P,x) = \sum_{y \in P - \{\ho\}} g([\hz,y],x) \cdot (x-1)^{n - \rho(y)}.
$$
\end{enumerate}
The vectors
$(h_0, \ldots, h_n)$ 
and
$(g_0, \ldots, g_m) = (h_0, h_1 - h_0, h_2 - h_1, \ldots, h_m - h_{m-1})$ 
where $m = \lfloor n/2 \rfloor$
are respectively 
the {\em toric $h$-vector} 
and
the {\em toric $g$-vector} 
of $P$.

In general, the toric $h$-vector and the toric $g$-vector depend on the
{\em flag vectors} of the poset, that is,
the number $f_{S}$ of chains $\hz < t_{1} < t_{2} < \cdots < t_{k} < \ho$ satisfying
$\rho(t_{i}) = s_{i}$ for each subset
$S=\{s_{1} < s_{2} < \cdots < s_{k}\} \subseteq \{1,2,\ldots,n\}$.
The first formulas expressing the toric $h$-vector in
terms of the flag vectors are due to Bayer and
Ehrenborg~\cite{Bayer_Ehrenborg}. Stanley~\cite[Section 2]{Stanley_toric_g}
cites Kalai for observing that in the case when the poset 
obtained by removing the minimum element of an Eulerian poset $P$,
that is,
the resulting poset
$P - \{\hz\}$ is
dual simplicial, the toric $h$-polynomial depends only on the face
numbers of the underlying simplicial complex. A straightforward expression was found
by Hetyei~\cite[Corollary~6.9]{Hetyei_toric}, which we will use; see
Lemma~\ref{lemma_gh}.

We warn the reader that in this subsection we have followed Stanley's notation.
However, from this point on in the paper we only use
the toric $g$-vector and not the toric $h$-vector
and hence our notation will be consistent. Furthermore, the $h$-vector
introduced in Section~\ref{section_the_toric_g-polynomial} is the usual
$h$-vector of a simple polytope, which is also the $h$-vector of its
simplicial dual. This $h$-vector is not the same as the toric $h$-vector
introduced above.

\subsection{Permutations}

For a non-negative integer $n$, let $[n]$ denote the set $\{1,2,\ldots,n\}$.
Similarly for two integers $i \leq j$, let $[i,j]$ denote the
interval $\{i,i+1, \ldots, j\}$.

Let $\SSSS_{n}$ denote the symmetric group on $n$ elements.
We choose the underlying set to be the set $[n]$.
The index $i \in [n-1]$ is an {\em ascent} of
a permutation $\pi = (\pi(1),\pi(2),\ldots,\pi(n)) \in \SSSS_{n}$
if $\pi(i) < \pi(i+1)$, otherwise $i$ is a {\em descent}.
The ascent set and descent set of a permutation $\pi$
are given by
\begin{align*}
\Asc(\pi)
& =
\{i \in [n-1] : \pi(i) < \pi(i+1)\} ,
&
\Des(\pi)
& =
\{i \in [n-1] : \pi(i) > \pi(i+1)\} .
\end{align*}
The number of ascents and descents are denoted by
$\asc(\pi) = |\Asc(\pi)|$,
respectively
$\des(\pi) = |\Des(\pi)|$.
Finally, a permutation
$\pi \in \SSSS_{n}$
is {\em $123$-avoiding} if there is no
triple $1 \leq i_{1} < i_{2} < i_{3} \leq n$ satisfying
$\pi(i_{1})<\pi(i_{2})<\pi(i_{3})$.
For more on pattern avoidance,
see~\cite{Kitaev}.

\subsection{Lattice paths}
\label{subsection_lattice_paths}
A {\em Dyck path} of semilength $n$
is a lattice path from the origin~$(0,0)$ to $(2n,0)$
composed of $n$ up steps~$(1,1)$ and $n$ down steps~$(1,-1)$
that does not go below the horizontal axis.  
We may encode such a Dyck path with the associated
{\em Dyck word} by recording each up step by the letter~$U$ and
each down step by the letter~$D$.  
The number of Dyck paths, and hence also the number Dyck words, of
semilength~$n$ is the $n$th Catalan number
\begin{align*}
C_{n}
& =
\frac{1}{n+1} \cdot \binomial{2n}{n} .
\end{align*}
The generating function for the Catalan numbers is well-known:
\begin{align}
\label{equation_Catalan_generating_function}  
C(t) & = \sum_{n \geq 0} C_{n} \cdot t^{n} = \frac{1-\sqrt{1-4t}}{2t}.
\end{align}  
By rotating our perspective by $\pi/4$ radians counterclockwise, we may
also consider paths
from the origin to the point~$(n,n)$ taking North steps $(0,1)$
and East steps $(1,0)$ that never go below the line $y=x$
as Dyck paths. Similarly, we also consider the mirror image of such a
path that never goes above the line $y=x$ as a Dyck path.
These types of Dyck paths will be useful in
Sections~\ref{section_associahedron}
and~\ref{section_cyclohedron}.

We say that a word $v$ is a {\em factor} of the word $w$
if $w$ can be factored as $w = u \cdot v \cdot z$.
For instance, {\em peaks} in a Dyck path correspond to
factors of the form $UD$, whereas {\em valleys} correspond
to factors~$DU$.
We also say that a $UD$-word is {\em balanced}
if it contains the same number of $U$ letters as $D$ letters.

\begin{figure}[t]
\begin{center}
\begin{tikzpicture}[scale = 0.5]
\draw[step=1,gray,very thin] (0,0) grid (10,10);
\draw[thick,-] (0,0) -- (4,0) -- (4,2) -- (6,2) -- (6,5) -- (9,5) -- (9,7) -- (10,7) -- (10,10);
\draw[dashed] (0,0) -- (10,10);
%%%% Labeling under
\node at (0.5,-0.5) {10};
\node at (1.5,-0.5) {9};
\node at (2.5,-0.5) {8};
\node at (3.5,-0.5) {7};
\node at (4.5,-0.5) {6};
\node at (5.5,-0.5) {5};
\node at (6.5,-0.5) {4};
\node at (7.5,-0.5) {3};
\node at (8.5,-0.5) {2};
\node at (9.5,-0.5) {1};
%%%% "Krattenthaler"
\node at (3.5,0.5) {7};
\node at (3.5,1.5) {10};
\node at (5.5,2.5) {5};
\node at (5.5,3.5) {9};
\node at (5.5,4.5) {8};
\node at (8.5,5.5) {2};
\node at (8.5,6.5) {6};
\node at (9.5,7.5) {1};
\node at (9.5,8.5) {4};
\node at (9.5,9.5) {3};
\draw (3.5,0.5) circle (0.45);
\draw (3.5,1.5) circle (0.45);
\draw (5.5,2.5) circle (0.45);
\draw (5.5,3.5) circle (0.45);
\draw (5.5,4.5) circle (0.45);
\draw (8.5,5.5) circle (0.45);
\draw (8.5,6.5) circle (0.45);
\draw (9.5,7.5) circle (0.45);
\draw (9.5,8.5) circle (0.45);
\draw (9.5,9.5) circle (0.45);
\end{tikzpicture}
\end{center}
\caption{The Krattenthaler Dyck path
$w = U^{4} D^{2} U^{2} D^{3} U^{3} D^{2} U D^{3}$
for the $123$-avoiding permutation
$\pi=(\underline{7},10,\underline{5},9,8,\underline{2},6,\underline{1},4,3)$
where the left-to-right minima are underlined.}
\label{figure_Krattenthaler}
\end{figure}
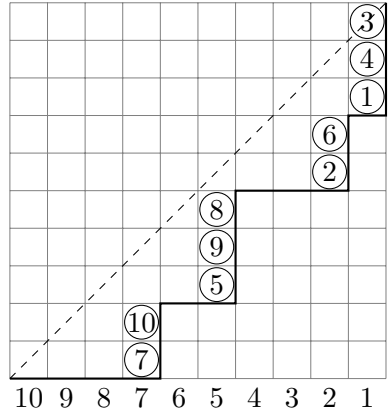

A key tool for our proofs will be a variant of the bijection between
$123$-avoiding permutations of the set $[n]$ and Dyck paths of semilength $n$
introduced by Krattenthaler~\cite{Krattenthaler}.
 An example of our bijection is shown in
Figure~\ref{figure_Krattenthaler}.  Note the lattice path never goes above
the line $y=x$. Given a $123$-avoiding permutation $\pi$ of $[n]$, we
rephrase Krattenthaler's bijection as  follows. 
\begin{enumerate}
\item Label the $U$-steps (these are the East steps in
  Figure~\ref{figure_Krattenthaler}) in decreasing order from $n$ to $1$.
\item For each left-to-right minimum $\pi(i)$, insert a $D$ step
  (North step) right after the $U$-step labeled $\pi(i)$ and label
  this $D$-step also with $\pi(i)$. By abuse of terminology we say
  that the peak is labeled $\pi(i)$. 
\item From each peak labeled $\pi(i)$, we continue with $D$ steps labeled
  $\pi(i+1), \pi(i+2),\ldots, \pi(j-1)$ where $\pi(j)$ is the next left-to-right minimum.
 \item After the $D$ step labeled $\pi(j-1)$, we continue with $U$ steps
   until we reach the next peak labeled~$\pi(j)$.  
\end{enumerate}  
Note that the complement of the left-to-right minima
of a $123$-avoiding permutation forms a decreasing sequence.
This observation is essential in the argument that the
inverse of Krattenthaler's bijection is well-defined.

We also consider lattice paths taking up steps~$(1,1)$ and down
steps~$(1,-1)$ that start at the origin and end at the lattice
point~$(n,n-2k)$ that do not go below the $x$-axis. Any such path must
contain $n-k$ up steps and $k$ down steps. They are enumerated by
\begin{align}
\label{equation_Catalan_triangle}
C(n,k)
& =
\binom{n}{k}-\binom{n}{k-1}\quad\text{for $0\leq k\leq \half{n}$}.
\end{align}  
The numbers $C(n,k)$ are the entries of the {\em Catalan triangle},
listed as sequence A008315 in OEIS~\cite{OEIS}. 

Finally we also consider the set of Dyck words of semilength $n$
that do not contain $UUU$ as a factor. The number of such words is known to
be the {\em Motzkin number}
\begin{align}
M_{n}
& =
\sum_{j=0}^{\lfloor n/2\rfloor}
\binom{n}{2j} \cdot C_{j} .
\label{equation_Motzkin}
\end{align}
By definition, the Motzkin number $M_{n}$ is the number of {\em Motzkin
  paths of length $n$} from $(0,0)$ to $(n,0)$ composed of up
steps~$(1,1)$, down steps~$(1,-1)$ and horizontal steps $(1,0)$ that
never go below the horizontal axis.
The parameter $j$ in~\eqref{equation_Motzkin}
is the number of up steps (and down steps) in the path.
A bijection between Motzkin paths of
length $n$ and $UUU$-avoiding Dyck paths of semilength $n$ is outlined
in the entry of sequence A001006 in OEIS~\cite{OEIS}:
replace each $UUD$ factor
in the Dyck word with an up step, then replace each remaining
$UD$ factor with a horizontal step, and finally the  remaining $D$ steps
are unchanged.
From this discussion we have part (a) of the following lemma.
\begin{lemma}
\begin{itemize}
\item[(a)]
The number of $UUU$-avoiding Dyck paths of semilength $n$
having exactly $j$ factors of $UU$ is given by
$\binom{n}{2j} \cdot C_{j}$.
\item[(b)]
The number of balanced
$UD$-words with $n$ $U$'s and $n$ $D$'s, having no $UUU$ factors,
ending with a $D$ and having exactly $j$ $UU$ factors is
$\binom{n}{2j} \cdot \binom{2j}{j}$.
\end{itemize}
\label{lemma_j_UU}
\end{lemma}
Part (b) of this lemma follows by the same bijection.
Instead of Motzkin paths we consider all possible
paths from $(0,0)$ to $(n,0)$ consisting of up, down
and horizontal steps with no restriction.

\subsection{Parking functions}
\label{subsection_parking_functions}

\begin{definition}
For a function $f: [n] \longrightarrow [n]$
we have the following three notions:
\begin{itemize}
\item[(i)]
The function $f$ is a {\em parking function}
if the fiber $f^{-1}([k])$ has at least $k$ elements
for $1 \leq k \leq n$.
\item[(ii)]
The function $f$ is {\em $123$-avoiding} if there is no triple
$1 \leq i_{1} < i_{2} < i_{3} \leq n$ such that $f(i_{1}) \leq f(i_{2}) \leq f(i_{3})$.
\item[(iii)]
The function $f$ has an {\em ascent} at position $i$ if
$1 \leq i \leq n-1$ and $f(i) \leq f(i+1)$.
\end{itemize}
\end{definition}
The study of parking functions originated with
Konheim and Weiss~\cite{Konheim_Weiss}.
Note that parts~(ii) and~(iii) of the above definition
are extensions of permutation notions to functions.

\begin{figure}[t]
\begin{center}
\begin{tikzpicture}[scale = 0.5]
\draw[step=1,gray,very thin] (0,0) grid (10,10);
\draw[thick,-] (10,10) -- (9,10) -- (9,9) -- (6,9) -- (6,7) -- (5,7) -- (5,6) -- (4,6) -- (4,5) -- (3,5) -- (3,4) -- (2,4) -- (2,2) -- (0,2) -- (0,0);
\draw[dashed] (0,0) -- (10,10);
%%%%
\node at (0.5,0.5) {7};
\node at (0.5,1.5) {10};
\node at (2.5,2.5) {5};
\node at (2.5,3.5) {9};
\node at (3.5,4.5) {8};
\node at (4.5,5.5) {2};
\node at (5.5,6.5) {6};
\node at (6.5,7.5) {1};
\node at (6.5,8.5) {4};
\node at (9.5,9.5) {3};

%%%% Labeling under
\node at (0.5,-0.5) {1};
\node at (1.5,-0.5) {2};
\node at (2.5,-0.5) {3};
\node at (3.5,-0.5) {4};
\node at (4.5,-0.5) {5};
\node at (5.5,-0.5) {6};
\node at (6.5,-0.5) {7};
\node at (7.5,-0.5) {8};
\node at (8.5,-0.5) {9};
\node at (9.5,-0.5) {10};

\draw[thick,->] (-12,5.5) node[above]{$1$} -- (-12,4.5) node[below]{$7$};
\draw[thick,->] (-11,5.5) node[above]{$2$} -- (-11,4.5) node[below]{$5$};
\draw[thick,->] (-10,5.5) node[above]{$3$} -- (-10,4.5) node[below]{$10$};
\draw[thick,->] (-9,5.5) node[above]{$4$} -- (-9,4.5) node[below]{$7$};
\draw[thick,->] (-8,5.5) node[above]{$5$} -- (-8,4.5) node[below]{$3$};
\draw[thick,->] (-7,5.5) node[above]{$6$} -- (-7,4.5) node[below]{$6$};
\draw[thick,->] (-6,5.5) node[above]{$7$} -- (-6,4.5) node[below]{$1$};
\draw[thick,->] (-5,5.5) node[above]{$8$} -- (-5,4.5) node[below]{$4$};
\draw[thick,->] (-4,5.5) node[above]{$9$} -- (-4,4.5) node[below]{$3$};
\draw[thick,->] (-3,5.5) node[above]{$10$} -- (-3,4.5) node[below]{$1$};
\node at (-12.9,5) {$f$:};
\end{tikzpicture}
\end{center}
\caption{The Dyck path of Garsia and Haiman  representing the parking
function $f = (7,5,10,7,3,6,1,4,3,1)$.
The associated permutation is
$\pi = (\underline{7},10,\underline{5},9,8,\underline{2},6,\underline{1},4,3)$
and the Dyck word is
$v = U^{2} D^{2} U^{2} D U D U D U D U^{2} D^{3} U D$.}
\label{figure_Garsia_Haiman}
\end{figure}
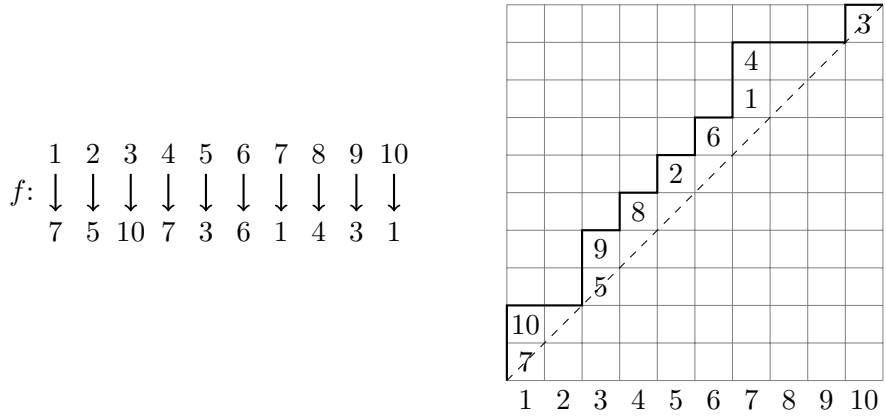

We say that a balanced word~$v$ of semilength $n$ and a permutation
$\pi \in \SSSS_{n}$ are {\em compatible} if
\begin{align*}
v
& =
U^{q_{1}} D
U^{q_{2}} D
\cdots
U^{q_{n}} D
&
\Des(\pi)
& \subseteq
\{q_{1}, q_{1}+q_{2}, \ldots, q_{1}+\cdots+q_{n-1}\} ,
\end{align*}
where $q_{1}, q_{2}, \ldots, q_{n}$
are nonnegative integers with sum $n$.
Garsia and Haiman~\cite[pages~226--227]{Garsia_Haiman} introduced a bijection between
functions $f : [n] \longrightarrow [n]$ and
the set
\begin{align*}
\{(\pi,v) : \pi \in \SSSS_{n}, v \text{ balanced $UD$-word},
\pi \text{ and } v \text{ compatible}\} .
\end{align*}
Given a function $f$, the balanced word $v$ is given by
$U^{q_{1}} D U^{q_{2}} D \cdots U^{q_{r}} D$
where $q_{i}$ is the cardinality of the fiber $f^{-1}(i)$,
that is, $q_{i} = |f^{-1}(i)|$.
Furthermore,
let $\tau_{i}$ be the list of the elements in the fiber $f^{-1}(i)$
in increasing order.
The permutation $\pi$ is then
the concatenation of the $n$ lists
$\tau_{1}$ through~$\tau_{n}$.

We visualize this bijection as follows;
see Figure~\ref{figure_Garsia_Haiman}.
We will denote each vertical step $(0,1)$ by the letter~$U$
and each horizontal step $(1,0)$ by the letter~$D$.
\begin{enumerate}
\item
Associate the element $i\in [n]$ in the range of
$f$ to the interval $(i-1,i)$ on the $x$-axis.
\item
Write the elements of the fiber $f^{-1}(i)$ in upward increasing
order as labels of the $U$
steps whose horizontal coordinate is $i-1$. These steps are
consecutive and we set the vertical coordinate of the first such $U$
step to be $|f^{-1}([i-1])|$.
\item Connect the resulting runs of $U$ steps with $D$ steps.
\item Finally, add enough $D$ steps after the last run of $U$ steps
so that the path ends at $(n,n)$.  
\end{enumerate}
Garsia and Haiman showed that a function $f$ is a
parking function if and only if the balanced word~$v$ is
a Dyck word.
In Figure~\ref{figure_Garsia_Haiman}
this is the lattice path occurring weakly above the line $y=x$.

\begin{definition}
Let $F = \{f_{0},f_{1},\ldots\}$ be an infinite alphabet and let the weight
of the letter $f_{q}$ be $w(f_{q}) = q-1$ for $q \geq 0$.
Furthermore, the weight of a word $v = v_{1} \cdots v_{k}$
with letters from this alphabet 
is the sum $w(v)=w(v_{1})+\cdots+w(v_{k})$.
A word $v = v_{1} \cdots v_{m}$ is a 
{\em \L ukasziewicz word}
if the weight of each initial factor
satisfies $w(v_{1} \cdots v_{k}) \geq 0$ for $k<m$ and $w(v)=-1$.  
\end{definition}  

\begin{lemma}
The map sending the word $U^{q_{1}} D U^{q_{2}} D \cdots
U^{q_{n}} D$ of length $2n$ to the word
$f_{q_{1}}f_{q_{2}}\cdots f_{q_{n}} f_0$ is a bijection between Dyck
words of length $2n$ and \L ukasziewicz words of length $n+1$.
\end{lemma}  
The straightforward verification is left to the reader.

\subsection{The restricted Foata--Strehl action}
\label{subsection_Foata_Strehl}

Some of our key results in Sections~\ref{section_permutahedron}
and~\ref{section_chordal_nestohedra} make use of a restriction of the
Foata--Strehl action~\cite{Foata_Strehl} on the set of permutations of
$[n]$ that was first considered by Br\"and\'en~\cite{Branden}, albeit in
a slightly modified form. 
We review these actions in terms of the {\em Foata--Strehl tree}
representation of the permutations of $[n]$.

\begin{definition}
\label{definition_increasing_vertex_labeling}  
A rooted tree on $n$ vertices has an {\em increasing vertex labeling}
if its vertices are bijectively labeled with the elements of an
$n$-element totally ordered set in such a way that the label on each
child is greater than the label on its parent.  A {\em Foata--Strehl
  tree on $n$ vertices} is a special rooted tree with an increasing
vertex labeling where every vertex has at most two children: at most
one left child and at most one right child.  
\end{definition}  

Figure~\ref{figure_Foata--Strehl_trees}
displays three Foata--Strehl trees.
These trees are called {\em increasing binary trees}
in~\cite[Section~1.5]{Stanley_EC1}.
We prefer the above terminology, because some sources
insist that every non-leaf vertex should have two children in a binary
tree, and they would call rooted trees of the above type
{\em $0$-$1$-$2$ trees}.
\begin{definition}
A $0$-$1$-$2$ tree is a rooted tree in which each vertex has at most $2$
children. 
\end{definition}  

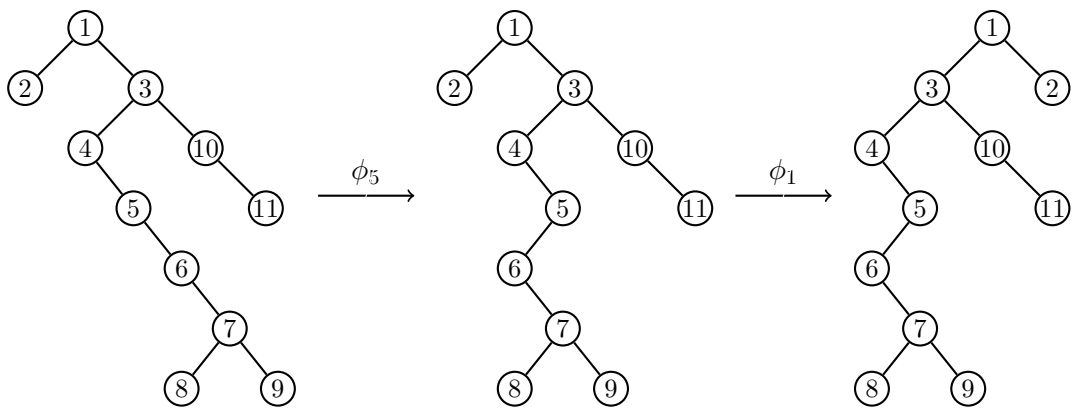
\begin{figure}[t]
\begin{center}
\tikzstyle{place}=[circle,draw=black,thick,
                   inner sep=0pt,minimum size=4.5mm]
\begin{tikzpicture}[scale = 0.8]
\node (one) at (0,6) [place] {\small ${1}$};
\node (two) at (-1,5) [place] {\small ${2}$};
\node (three) at (1,5) [place] {\small ${3}$};
\node (four) at (0,4) [place] {\small ${4}$};
\node (ten) at (2,4) [place] {\small ${10}$};
\node (five) at (0.8,3) [place] {\small ${5}$};
\node (six) at (1.6,2) [place] {\small ${6}$};
\node (eleven) at (3,3) [place] {\small ${11}$};
\node (eight) at (1.6,0) [place] {\small ${8}$};
\node (nine) at (3.2,0) [place] {\small ${9}$};
\node (seven) at (2.4,1) [place] {\small ${7}$};

\draw[-,thick] (one) -- (two);
\draw[-,thick] (one) -- (three);
\draw[-,thick] (three) -- (four);
\draw[-,thick] (three) -- (ten);
\draw[-,thick] (four) -- (five) -- (six) -- (seven) -- (nine);
\draw[-,thick] (ten) -- (eleven);
\draw[-,thick] (seven) -- (eight);
\end{tikzpicture}
\begin{tikzpicture}[scale = 0.8]
\draw[->,thick] (-0.8,0) --node[above]{$\phi_{5}$} (0.8,0);
\draw[rectangle,white] (-1,-3.5) -- (1,3.5);
\end{tikzpicture}
\begin{tikzpicture}[scale = 0.8]
\node (one) at (0,6) [place] {\small ${1}$};
\node (two) at (-1,5) [place] {\small ${2}$};
\node (three) at (1,5) [place] {\small ${3}$};
\node (four) at (0,4) [place] {\small ${4}$};
\node (ten) at (2,4) [place] {\small ${10}$};
\node (five) at (0.8,3) [place] {\small ${5}$};
\node (six) at (0,2) [place] {\small ${6}$};
\node (eleven) at (3,3) [place] {\small ${11}$};
\node (eight) at (0,0) [place] {\small ${8}$};
\node (nine) at (1.6,0) [place] {\small ${9}$};
\node (seven) at (0.8,1) [place] {\small ${7}$};

\draw[-,thick] (one) -- (two);
\draw[-,thick] (one) -- (three);
\draw[-,thick] (three) -- (four);
\draw[-,thick] (three) -- (ten);
\draw[-,thick] (four) -- (five) -- (six) -- (seven) -- (nine);
\draw[-,thick] (ten) -- (eleven);
\draw[-,thick] (seven) -- (eight);
\end{tikzpicture}
\begin{tikzpicture}[scale = 0.8]
\draw[->,thick] (-0.8,0) --node[above]{$\phi_{1}$} (0.8,0);
\draw[rectangle,white] (-1,-3.5) -- (1,3.5);
\end{tikzpicture}
\begin{tikzpicture}[scale = 0.8]
\node (one) at (2,6) [place] {\small ${1}$};
\node (two) at (3,5) [place] {\small ${2}$};
\node (three) at (1,5) [place] {\small ${3}$};
\node (four) at (0,4) [place] {\small ${4}$};
\node (ten) at (2,4) [place] {\small ${10}$};
\node (five) at (0.8,3) [place] {\small ${5}$};
\node (six) at (0,2) [place] {\small ${6}$};
\node (eleven) at (3,3) [place] {\small ${11}$};
\node (eight) at (0,0) [place] {\small ${8}$};
\node (nine) at (1.6,0) [place] {\small ${9}$};
\node (seven) at (0.8,1) [place] {\small ${7}$};

\draw[-,thick] (one) -- (two);
\draw[-,thick] (one) -- (three);
\draw[-,thick] (three) -- (four);
\draw[-,thick] (three) -- (ten);
\draw[-,thick] (four) -- (five) -- (six) -- (seven) -- (nine);
\draw[-,thick] (ten) -- (eleven);
\draw[-,thick] (seven) -- (eight);
\end{tikzpicture}
\end{center}
\caption{The Foata--Strehl trees for the permutations
$\tau = (2$, $1$, $4$, $5$, $6$, $8$, $7$, $9$, $3$, $10$, $11)$,
$\phi_{5}\tau$ and $\phi_{1}\phi_{5}\tau$.}
\label{figure_Foata--Strehl_trees}
\end{figure}

Even after disregarding the vertex labeling, a Foata--Strehl tree is not
immediately identifiable with a plane $0$-$1$-$2$ tree, because there is
no distinction between a right child and a left child in a plane tree
if the child is the only child of its parent. In this paper we will use
the following identification.

\begin{definition}
\label{definition_right-adjusted}  
A Foata--Strehl tree is {\em right-adjusted} if no vertex has
an only child that is a left child. Given a plane $0$-$1$-$2$ tree 
with an increasing vertex labeling, we identify it with
the right-adjusted Foata--Strehl tree obtained by designating each only
child to be a right child. 
\end{definition}

Clearly there is a bijection between labeled plane $0$-$1$-$2$ trees
with an increasing vertex labeling and right-adjusted Foata--Strehl
trees. As explained in~\cite[Section~1.5]{Stanley_EC1}, there is a bijection
between Foata--Strehl trees with $n$-vertices and permutations of $[n]$,
defined as follows.

\begin{definition}
Given a permutation $\pi=\pi(1)\cdots\pi(n)$ of letters of an $n$-element
totally ordered set, we define its Foata--Strehl tree recursively as
follows:
\begin{enumerate}
\item Label the root of the tree with $\pi(m)=\min
  \{\pi(1),\ldots,\pi(n)\}$.
\item Define the left, respectively right subtree of the root as the
  Foata--Strehl tree of the permutation $\pi(1)\cdots\pi(m-1)$,
  respectively $\pi(m+1)\cdots \pi(n)$. (Either or both subtrees may be
  empty.) The root of the left
  (respectively right) subtree is the left (respectively right) child of
  the vertex labeled $\pi(m)$.   
\end{enumerate}
\end{definition}  

Conversely, the permutation associated to a Foata--Strehl tree is
obtained by reading off the labels on the vertices using the {\em
  inorder traversal}. This recursively defined process calls for reading
the labels in the left subtree of the root, then the label of the root,
and then the labels in the right subtree of the root. For example the
Foata--Strehl tree shown on the left-hand side of
Figure~\ref{figure_Foata--Strehl_trees} represents the permutation
$(2,1,4,5,6,8,7,9,3,10,11)$.  

The {\em Foata--Strehl action} is most easily defined on the
Foata--Strehl trees representing the permutations: the operation $\phi_x$
exchanges the left and right subtrees of the vertex labeled $x$. For
example
$$
\phi_{1}\phi_{5}(2,1,4,5,6,8,7,9,3,10,11)
=\phi_{1} (2,1,4,6,8,7,9,5,3,10,11)
=(4,6,8,7,9,5,3,10,11,1,2),
$$
where the associated trees are displayed in
Figure~\ref{figure_Foata--Strehl_trees}.
Clearly $\phi_{i}$ and $\phi_{j}$ commute.
Hence the Foata--Strehl action is a
${\mathbb Z}_{2}^{n-1}$-action on the set of all permutations of
the set $[n]$.
The numbers of the orbits are the tangent and secant numbers, and
{\em Andr\'e permutations} (of various kinds) may be used as a set of distinct
orbit representatives.

Motivated by Br\"and\'en's definition~\cite{Branden}, we define the 
{\em restricted Foata--Strehl action} $\psi$ as follows: set $\psi_x=\phi_x$
if the vertex labeled $x$ has exactly one child, and set $\psi_x$ to
be the identity operation otherwise. The permutations represented
by right-adjusted Foata--Strehl trees are a natural choice of distinct orbit
representatives for this action. Br\"and\'en's {\em modified Foata--Strehl
  action} is obtained from this restricted action by replacing each
label $i$ with $n+1-i$, thus turning the Foata--Strehl trees into {\em
  decreasing trees}: the label on each child is less than the label on
its parent. 
Br\"and\'en's paper avoids introducing tree representations.
It instead operates with descents and ascents.
\begin{definition}
Let $\pi=\pi(1)\cdots\pi(n)$ be a permutation of $[n]$. The index $i<n$
is a {\em descent}, respectively an {\em ascent}, if $\pi(i)>\pi(i+1)$,
respectively $\pi(i)<\pi(i+1)$ holds. The index $2 \leq i \leq n-1$
is a {\em peak}, a {\em valley},
a {\em double descent}, or a {\em double ascent},
respectively, if $\pi(i-1)<\pi(i)>\pi(i+1)$,
$\pi(i-1)>\pi(i)<\pi(i+1)$,  $\pi(i-1)>\pi(i)>\pi(i+1)$, or
$\pi(i-1)<\pi(i)<\pi(i+1)$ hold, respectively.      
\end{definition}  

The equivalence between the above
construction and Br\"and\'en's is straightforward in light of the
following facts; see~\cite[Section~1.5]{Stanley_EC1}.
\begin{lemma}\label{lemma_children} Consider a permutation
  $\pi=\pi(1)\cdots\pi(n)$. The index 
  $2 \leq i \leq n-1$ is a peak,   a valley, a double descent or a
  double ascent, respectively, if the   vertex labeled $\pi(i)$ in its
  Foata--Strehl tree has no children, two children, only a left child,
  only a right child, respectively. 
\end{lemma}
Lemma~\ref{lemma_children} implies the following characterization of the
orbit representatives of the restricted Foata--Strehl action. 
\begin{corollary}
\label{corollary_right-adjusted}  
The Foata--Strehl tree of a permutation $\pi$ is right-adjusted if and
only if the permutation $\pi$ has no double descents and no final
descent.  
\end{corollary}

\section{The toric $g$-polynomial in terms of the $\gamma$-vector}
\label{section_the_toric_g-polynomial}

Recall an $n$-dimensional polytope $P$ is {\em simple}
(or {\em dual simplicial})
if every vertex is incident to $n$ facets (maximal faces).
Let $f_{i} = f_{i}(P)$ be the number of $i$-dimensional faces
of the polytope $P$ for $0 \leq i \leq n$.
The associated $h$-polynomial is given by
\begin{align}
\label{equation_f_to_h_vector}
\sum_{i=0}^{n} h_{i} \cdot x^{i}
& =
\sum_{i=0}^{n} f_{i} \cdot (x-1)^{i} .
\end{align}
The Dehn-Sommerville relations state that this polynomial is palindromic,
that is, $h_{i} = h_{n-i}$ for $0 \leq i \leq n$.
The {\em $\gamma$-vector}
$\gamma = (\gamma_{0}, \gamma_{1}, \ldots, \gamma_{\lfloor\half{n}\rfloor})$
of an $n$-dimensional simple polytope 
was introduced by Gal~\cite[Definition 2.1.4]{Gal} and is given by
\begin{align}
\label{equation_gamma_vector}
\sum_{i=0}^{n} h_{i} \cdot x^{i}
& =
\sum_{j=0}^{\left\lfloor\half{n}\right\rfloor} \gamma_{j}  \cdot x^{j} \cdot (1+x)^{n-2j}.   
\end{align}
The fact that the $\gamma$-vector is well-defined follows
from the Dehn-Sommerville relations.

Consider the coefficient of $x^{i}$ in equation~\eqref{equation_gamma_vector}.
For $0\leq i \leq \half{n}$, we have
\begin{align*}
h_{i}
& =
\sum_{j=0}^{i} \binom{n-2j}{i-j} \cdot \gamma_{j} .
\end{align*}
Taking the backward difference, we obtain
\begin{align}
h_{i}-h_{i-1}
& =
\gamma_{i}+\sum_{j=0}^{i-1}
\left(\binom{n-2j}{i-j}-\binom{n-2j}{i-j-1}\right) \cdot \gamma_{j}
\nonumber \\
& =
\sum_{j=0}^{i} C(n-2j,i-j) \cdot \gamma_{j} ,
\label{equation_hdiff2}  
\end{align}
for $1\leq i\leq \half{n}$, where we used
the entries of the Catalan triangle~\eqref{equation_Catalan_triangle}.
For a lattice path interpretation of the $h$-vector difference
$h_i - h_{i-1}$, see
work of Nevo--Peterson--Tenner~\cite[Observation~6.1]{Nevo_Petersen_Tenner}.
We note the paper~\cite{Nevo_Petersen_Tenner} is closely related to
the Nevo--Petersen conjecture~\cite[Conjecture 1.4]{Nevo_Petersen}
and Conjecture~\ref{conjecture_simplicial} in the introduction.
See also~\cite{Brittenham_Carroll_Petersen_Thomas}
for examples inspired by $q$-analogues.

Introducing $C(n,0,x)=1$ and
\begin{align}
\label{equation_Catalan_triangle_x}  
C(n,i,x)
& =
\sum_{k=1}^{i}
\frac{n+1-2i}{k} \cdot \binom{n-i}{k-1} \cdot \binom{i-1}{k-1} \cdot x^{k}
\end{align}
for $1\leq i \leq \half{n}$,
we may write the toric
$g$-polynomial of an $n$-dimensional simple polytope $P$ in terms of its
$h$-vector~\cite[Corollary~6.9]{Hetyei_toric} 
as follows: 
\begin{lemma}[Hetyei]
\label{lemma_gh}
The toric $g$-polynomial of an $n$-dimensional simple polytope $P$ is
given by 
\begin{align}
\label{equation_gh}
g(P,x)
& =
h_{0} + \sum_{i=1}^{\left\lfloor \half{n}\right\rfloor} (h_{i}-h_{i-1}) \cdot C(n,i,x).
\end{align}
\end{lemma}
Here we used the fact that the $h$-vector of a simple polytope is the
same as the $h$-vector of the boundary complex of its simplicial dual. 
Combining Lemma~\ref{lemma_gh} with equation~\eqref{equation_hdiff2}
(and taking into account 
$\gamma_{0}=h_{0}$), 
we obtain 
\begin{align}  
g(P,x)
& = 
\gamma_{0}+\sum_{i=1}^{\left\lfloor \half{n}\right\rfloor}
\sum_{j=0}^{i}  C(n-2j,i-j) \cdot C(n,i,x) \cdot \gamma_{j}
\nonumber \\
& =
\sum_{i=0}^{\left\lfloor \half{n}\right\rfloor}\sum_{j=0}^{i}
C(n-2j,i-j) \cdot C(n,i,x) \cdot \gamma_{j}
\nonumber \\
&=\sum_{j=0}^{\left\lfloor \half{n}\right\rfloor} \gamma_{j} \cdot
\sum_{i=j}^{\left\lfloor \half{n}\right\rfloor} C(n-2j,i-j) \cdot C(n,i,x) .
\label{equation_g_gamma}  
\end{align}
Krattenthaler's remark~\cite[Remark 6.12]{Hetyei_toric} may be
rephrased as follows.
\begin{proposition}[Krattenthaler]
\label{proposition_Krattenthaler}
For $1\leq i\leq \half{n}$, the coefficient of $x^{k}$ in the polynomial
$C(n,i,x)$ is the number of Dyck paths from $(0,0)$ to $(n,n-2i)$
having $k$ peaks. 
\end{proposition}
\begin{proof}
Reflecting the Dyck path over the line $y=x$ and then
rotating it clockwise by $\pi/4$ radians
in~\cite[Remark 6.12]{Hetyei_toric},
Krattenthaler's remark can be restated
as the sought after coefficient
is the number of all Dyck paths from $(0,0)$ to $(n+1,n+1-2i)$
containing $k$ valleys such that there are no additional down steps after the
last valley. Such a lattice path necessarily has also $k$ peaks, and the
last step must be an up step. Removing the final up step yields a Dyck
path from $(0,0)$ to $(n,n-2i)$ having $k$ peaks (and $k$ or $k-1$
valleys). The inverse of this operation is adding a final up step to Dyck
paths from $(0,0)$ to $(n,n-2i)$ having $k$ peaks. In the resulting
lattice path there are $k$ valleys, and there are no peaks after the
last valley. 
\end{proof}  

\begin{lemma}
\label{lemma_peaks}  
The coefficient of $x^{k}$ in
$
\sum_{i=j}^{\left\lfloor \half{n}\right\rfloor}
  C(n-2j,i-j)
\cdot C(n,i,x)$
is the number of Dyck paths of semilength $n-j$ with $k$ peaks whose
first coordinate is at most $n-1$.
\end{lemma}
\begin{proof}
Let
$(0,0)=(x_{0},y_{0}), (x_{1},y_{1}),\ldots, (x_{2n-2j},y_{2n-2j})=(2n-2j,0)$
be any Dyck path of semilength $n-j$.
Let  $i=\half{(n-y_{n})}$. Note that we must be able to return to level $0$
from level $y_{n}$ in $n-2j$ steps. Hence $0\leq y_{n}\leq n-2j$ and the
parity of $y_{n}$ is the same as the parity of $n-2j$.  In turn, this is the same
as the parity of $n$. We obtain that $i$ is an integer satisfying $j\leq
i\leq \half{n}$ and that $y_{n}=n-2i$. After fixing~$i$,
the number of possible lattice paths $(0,0)=(x_{0},y_{0}), (x_{1},y_{1}),\ldots,
(x_{n},y_{n})=(n,n-2i)$ from $(0,0)$ to $(n,2n-i)$ with $k$ peaks
is the coefficient of $x^{k}$ in $C(n,i,x)$. Furthermore, there are
$C(n-2j,i-j)$ ways to select the lattice path $(n,n-2i)=(x_{n},y_{n}),
(x_{n+1},y_{n+1}),\ldots, (x_{2n-2j},y_{2n-2j})=(2n-2j,0)$.  
\end{proof}  

Define the {\em toric $g$-contribution polynomial} $g_{n,j}(x)$ as follows:
\begin{align}
\label{equation_gp}
g_{n,j}(x)
& =
\sum_{k=0}^{\min\left(\left\lfloor \half{n}\right\rfloor, n-j\right)}
C_{n-k-j} \cdot \binom{n-k}{k} \cdot (x-1)^{k} ,
\end{align}  
for $0 \leq j \leq n$.
We remark that
the polynomial $g_{n,0}(x)$
is the toric $g$-polynomial of the $n$-dimensional
cube~\cite{Hetyei_second_look}.
The next result shows that for $0 \leq j \leq \half{n}$
the polynomial $g_{n,j}(x)$ is the contribution of
the $\gamma$-vector entry $\gamma_{j}$
to the toric $g$-polynomial of an $n$-dimensional simple polytope.
\begin{theorem}
\label{theorem_peaks}  
The toric $g$-polynomial of an $n$-dimensional simple polytope $P$
is given by   
\begin{align}
g(P,x)
& =
\sum_{j=0}^{\left\lfloor \half{n}\right\rfloor} \gamma_{j} \cdot g_{n,j}(x) ,
\end{align}
where
$(\gamma_{0},\gamma_{1},\ldots,\gamma_{\left\lfloor \half{n}\right\rfloor})$
is the $\gamma$-vector of the simple polytope $P$.   
\end{theorem}  
\begin{proof}
Replacing $x$ by $x+1$ in
equations~\eqref{equation_g_gamma} and~\eqref{equation_gp}
and fixing $j$,
we need to show the following equivalent identity:
\begin{align*}
\sum_{i=j}^{\left\lfloor \half{n}\right\rfloor} C(n-2j,i-j) \cdot C(n,i,x+1)
& =
\sum_{k=0}^{\left\lfloor \half{n}\right\rfloor} C_{n-j-k} \cdot \binom{n-k}{k} \cdot x^{k}.
\end{align*}
By Lemma~\ref{lemma_peaks} the contribution of $\gamma_{j}$ to $g(P,x)$ is
the total weight of Dyck paths of semilength $n-j$ where each peak whose
first coordinate is in the interval $[1,n-1]$ contributes a factor of
$x$ to the weight of the lattice path. When we substitute $x+1$ into $x$,
this corresponds to marking some of these peaks and we count only the
marked peaks in the exponent of $x$.

Subject to this rephrasing, it suffices to show the following statement:
The number of Dyck paths from $(0,0)$ to $(2n-2j,0)$ having $k$
marked peaks with first coordinate at most $n-1$ is
$C_{n-k-j} \cdot \binom{n-k}{k}$.
Let $w = w_{1}w_{2}\cdots w_{2n-2j}$ be the associated Dyck word.

Hence we
are counting Dyck words of semilength $n-j$ with some marked factors~$UD$
among the first $n$ letters. Let us remove these factors. If we
remove $k$ such marked factors from a Dyck word of semilength
$n-j$, we obtain a Dyck word $v = v_{1} v_{2} \cdots v_{2n-2j-2k}$
of semilength $n-j-k$. Note that the marked peaks belong to the interval
$[1,n-1]$ if and only if all marked factors occurred before the
letter $v_{n-2k+1}$. There are $C_{n-j-k}$ ways to select the Dyck
word $v$. To reconstruct $w$ we must insert $k$
factors before $v_{n-2k+1}$. In other words, we must line up the letters $v_{1},
v_{2},\ldots, v_{n-2k}$ (in this relative order) and $k$ factors
$UD$. This may be performed in $\binom{n-2k+k}{k}=\binom{n-k}{k}$ ways.
\end{proof}  

\begin{lemma}
Let $P$ be an $n$-dimensional simple polytope
with $\gamma$-vector
$(\gamma_{0}, \gamma_{1}, \ldots, \gamma_{\lfloor \half{n} \rfloor})$.
Then the toric $g$-polynomial of the polytope $P$ is given by
\begin{align*}
g(P,x)
& =
\sum_{k=0}^{\left\lfloor \half{n}\right\rfloor} \binom{n-k}{k} \cdot (x-1)^{k} \cdot
\sum_{j=0}^{\left\lfloor\half{n}\right\rfloor} \gamma_{j} \cdot C_{n-k-j} .
\end{align*}
\label{lemma_explicit_g_polynomial}
\end{lemma}
\begin{proof}
We have
\begin{align*}
g(P,x)
& =
\sum_{j=0}^{\left\lfloor\half{n}\right\rfloor}
\gamma_{j} \cdot g_{n,j}(x) \\
& =
\sum_{j=0}^{\left\lfloor\half{n}\right\rfloor} \gamma_{j} \cdot 
\sum_{k=0}^{\left\lfloor \half{n}\right\rfloor} C_{n-k-j} \cdot \binom{n-k}{k} \cdot (x-1)^{k}\\
& =
\sum_{k=0}^{\left\lfloor \half{n}\right\rfloor} \binom{n-k}{k} \cdot (x-1)^{k} \cdot
\sum_{j=0}^{\left\lfloor\half{n}\right\rfloor} \gamma_{j} \cdot C_{n-k-j} .
\qedhere
\end{align*}
\end{proof}

Next we turn our attention to the leading coefficient
of the toric $g$-polynomial.

\begin{corollary}
\label{corollary_2m_m}
Let $P$ be a $2m$-dimensional simple polytope
with $\gamma$-vector
$(\gamma_{0},\gamma_{1},\ldots,\gamma_{m})$.
The leading coefficient of the toric $g$-polynomial of
the polytope $P$ is given by   
\begin{align*}
[x^m] g(P,x)
& =
\sum_{j=0}^{m} \gamma_{j} \cdot C_{m-j} .
\end{align*}
Similarly, let $P$ be a $(2m+1)$-dimensional simple polytope
with $\gamma$-vector
$(\gamma_{0},\gamma_{1},\ldots,\gamma_{m})$.
The leading coefficient of the toric $g$-polynomial of
the polytope $P$ is given by   
\begin{align*}
[x^m] g(P,x)
& =
(m+1) \cdot \sum_{j=0}^{m} \gamma_{j} \cdot C_{m+1-j} .
\end{align*}
\end{corollary}
\begin{proof}
The only contribution to the toric $g$-polynomial
of $P$ of degree at least $m$ is the term $k=m$
in Lemma~\ref{lemma_explicit_g_polynomial},
and since $[x^{m}] (x-1)^{m} = 1$, the two results follow.
\end{proof}

\section{Compatible Dyck words}
\label{section_compatible}

In this section we give several interpretations of the coefficients of 
the toric $g$-contribution polynomial~$g_{n,j}(x)$.
The two first interpretations will be used in
Sections~\ref{section_associahedron}
and~\ref{section_cyclohedron}.

Recall that a set of integers is {\em
sparse} if it does not contain a pair of consecutive integers. 
Furthermore, a {\em labeled Dyck word $w$} of semilength $n$
is a Dyck word 
where the letters~$U$, respectively the letters~$D$,
are indexed from left to right by $1$ through~$n$.
For instance,
the Dyck word $U_{1}U_{2}D_{1}U_{3}D_{2}D_{3}U_{4}D_{4}$
is labeled.
We use the convention that a letter has no index if  
the index of a letter is irrelevant at the moment.

\begin{definition}
Let $A$ and $B$ be two sparse subsets of the set
$[n-1]$ and let $w$ be a labeled Dyck word of semilength $n$.
We say that the word $w$ is $(A,B)$-compatible if 
\begin{itemize}
\item[(a)]
for each $a \in A$ the word $U_{a}U_{a+1}D$ is a factor in $w$, and
\item[(b)]
for each $b \in B$ the word $UD_{b}D_{b+1}$ is a factor in $w$.
\end{itemize}
\end{definition}
Let ${\mathcal D}_{n}$ denote the set
of all Dyck words of semilength $n$.
Furthermore, let ${\mathcal D}_{n,A,B}$ be the set of all Dyck words of semilength $n$ that are $(A,B)$-compatible.
Observe that if a Dyck word
is both
$(A_{1},B_{1})$-compatible
and
$(A_{2},B_{2})$-compatible
then it is also
$(A_{1} \cup A_{2},B_{1} \cup B_{2})$-compatible.
\begin{theorem}
\label{theorem_compatible}
The number of Dyck words of semilength $n$ that are $(A,B)$-compatible
is given by the Catalan number $C_{n-|A|-|B|}$.
\end{theorem}
\begin{proof}
We present an explicit bijection 
$\phi_{n,A,B}: {\mathcal D}_{n,A,B} \longrightarrow {\mathcal D}_{n-|A|-|B|}$.
Let $w$ be a labeled Dyck word in the set ${\mathcal D}_{n,A,B}$.
First replace each occurrence of factors of the form
$U_{a}U_{a+1}D_{b}D_{b+1}$
where $a \in A$ and $b \in B$ by the empty word~$1$.
Note that this is the only way the two factors of the form
$U_{a}U_{a+1}D$ and $UD_{b}D_{b+1}$ can overlap.
Second, replace each remaining factor of the form
$U_{a}U_{a+1}D$ where $a \in A$ by $U$.
Finally, replace each remaining factor of the form
$UD_{b}D_{b+1}$ where $b \in B$ by $D$.
The resulting word is a Dyck word of semilength $n-|A|-|B|$.

The idea for establishing the inverse map of $\phi_{n,A,B}$ is to consider the
left-most replacement.
Let $a$, respectively $b$, be the minimal elements of the two sparse sets $A$ and $B$.
If the left-most replacement was $U_{a}U_{a+1}D \longmapsto U$
then the letter $U_{a}$ will precede the letter $D_{b-1}$
in the relabeled word.
Similarly, if the left-most replacement was $UD_{b}D_{b+1} \longmapsto D$
then $D_{b}$ will precede $U_{a-1}$.
Finally, if the replacement was
$U_{a}U_{a+1}D_{b}D_{b+1} \longmapsto 1$
then the two letters
$U_{a-1}$ and $D_{b-1}$ precede
the two letters
$U_{a}$ and $D_{b}$ in the image of $w$.
These are three distinct cases.

By induction on $|A|+|B|$ we construct the inverse map
$\psi_{n,A,B}: {\mathcal D}_{n-|A|-|B|} \longrightarrow {\mathcal D}_{n,A,B}$.
When $|A|+|B|=0$, the two sets
${\mathcal D}_{n-|A|-|B|}$ and ${\mathcal D}_{n,A,B}$
are equal and there is nothing to prove.
Let $w$ be a labeled Dyck word in ${\mathcal D}_{n-|A|-|B|}$.
Note in what follows the Dyck word
$\psi_{n,A,B}(w)$ is constructed from the word $w$ by reading
$w$ left to right.

Next we study the case when the set $A$ is empty
and $B$ is nonempty. 
Let $b$ be the minimal element of the set $B$.
Replace the occurrence of $D_{b}$ in $w$
with $UD_{b}D_{b+1}$. After relabeling the indices we obtain
a Dyck word $v$ in ${\mathcal D}_{n-|B|+1}$
Now set
$\psi_{n,\emptyset,B}(w) = \psi_{n,\emptyset,B-\{b\}}(v)$.
Observe that by induction this word
is $(\emptyset,B-\{b\})$-compatible.
Furthermore, by the construction it is also 
is $(\emptyset,\{b\})$-compatible.
In conclusion, it is $(\emptyset,B)$-compatible,
completing this case.

The case when $B$ is empty and $A$ is non-empty is
similar.

Now consider the case when $A$ and $B$ are both nonempty.
Let $a = \min(A)$ and $b = \min(B)$.
Consider the relative order of the four letters
$U_{a-1}$, $U_{a}$, $D_{b-1}$ and $D_{b}$
in the labeled word $w$. 
If $a=1$ or $b=1$ view the two non-existent letters
$U_{0}$ and $D_{0}$ as standing in front of the word $w$.
There are six possible arrangements that we arrange
in three cases, one of which has four subcases:
\begin{itemize}
\item[--]
The two letters $U_{a-1}$ and $D_{b-1}$ appear in $w$
before the two letters $U_{a}$ and $D_{b}$, reading left to right.
That is, we have one of the following four subcases:
\begin{align*}
U_{a-1} \cdots D_{b-1} \cdots U_{a} \cdots D_{b}, \\
U_{a-1} \cdots D_{b-1} \cdots D_{b} \cdots U_{a}, \\
D_{b-1} \cdots U_{a-1} \cdots U_{a} \cdots D_{b}, \\
D_{b-1} \cdots U_{a-1} \cdots D_{b} \cdots U_{a}.
\end{align*}
Furthermore, note that the second and third letters in each of these
subcases have to be adjacent, since no $U$ can be between
$U_{a-1}$ and $U_{a}$, and similarly, no letter $D$ between $D_{b-1}$
and $D_{b}$.
Hence we insert the factor $U_{a}U_{a+1}D_{b}D_{b+1}$ between the second and third letter
and relabel to obtain the Dyck word $v$.
Now define the map $\psi$ by
$\psi_{n,A,B}(w) = \psi_{n,A-\{a\},B-\{b\}}(v)$.

\item[--]
The two letters $U_{a-1}$ and $U_{a}$ appear in $w$
before the two letters $D_{b-1}$ and $D_{b}$ appear.
\begin{align*}
U_{a-1} \cdots U_{a} \cdots D_{b-1} \cdots D_{b}.
\end{align*}
Replace $U_{a}$ with $U_{a}U_{a+1}D$ and relabel to obtain the Dyck word $v$.
Now the map $\psi$ is given by
$\psi_{n,A,B}(w) = \psi_{n,A-\{a\},B}(v)$.

\item[--]
The last case is that the letters $D_{b-1}$ and $D_{b}$ appear
to the left of the letters $U_{a-1}$ and $U_{a}$,
that is, we have
\begin{align*}
D_{b-1} \cdots D_{b} \cdots U_{a-1} \cdots U_{a}.
\end{align*}
Similar to the previous case,
replace $D_{b}$ with $UD_{b}D_{b+1}$ and relabel to obtain the Dyck word~$v$
and set
$\psi_{n,A,B}(w) = \psi_{n,A,B-\{b\}}(v)$.
\end{itemize}
In each of these three cases observe that
the calculation of $\psi_{n,A',B'}(v)$
will not change the initial segment of the word $v$.
Hence since $v$ is
$(\{a\},\{b\})$-compatible, 
$(\{a\},\emptyset)$-compatible,
and
$(\emptyset,\{b\})$-compatible in each respective case,
we obtain that the image
$\psi_{n,A',B'}(v)$ is also so.
Thus the Dyck word $\psi_{n,A',B'}(v)$ is $(A,B)$-compatible.
By this construction we have that  $\psi_{n,A,B}$ is
the inverse of $\phi_{n,A,B}$
and the result follows.
\end{proof}

By using the same proof ideas we also obtain
the corresponding result for balanced words.
\begin{proposition}
\label{proposition_compatible_balanced}
The number of balanced words of semilength $n$ that are $(A,B)$-compatible
is given by the central binomial coefficient 
$\binom{2 \cdot(n-|A|-|B|)}{n-|A|-|B|}$.
\end{proposition}

Using Krattenthaler's bijection between $123$-avoiding
permutations on the set $[n]$ and Dyck paths of semilength~$n$ presented in
Subsection~\ref{subsection_lattice_paths},
Theorem~\ref{theorem_compatible} has a remarkable
consequence for the ascent statistics of $123$-avoiding permutations. 
A $123$-avoiding permutation $\pi$ has a sparse ascent set,
and the same holds for the inverse permutation~$\pi^{-1}$,
which is also $123$-avoiding.  
\begin{corollary}
Let $A$ and $B$ be sparse subsets of $[n-1]$. Then the number of
$123$-avoiding permutations~$\pi$ in $\SSSS_{n}$ such that 
$B \subseteq \Asc(\pi)$ and $A \subseteq \Asc(\pi^{-1})$
is given by the Catalan number $C_{n-|A|-|B|}$. 
\end{corollary}
For a $123$-avoiding permutation $\pi$,
it is easy to verify that $i$ is an ascent of $\pi$ if and only
if $\pi(i)$ is a left-to-right minimum and $\pi(i+1)$ is not. Equivalently,
in the corresponding Dyck path, the peak $UD$ labeled $\pi(i)$ must be part of
a factor $UDD$. Similarly, $j=\pi(i)$ is an ascent of the inverse~$\pi^{-1}$ if and
only if $j+1$ precedes $j$ in the word $\pi(1)\cdots \pi(n)$.
This happens if and only $j=\pi(i)$ is a left-to-right maximum and $j+1$ is not.
Equivalently, in the corresponding Dyck path, the peak $UD$ labeled
$\pi(i)$ must be part of a factor~$UUD$. 

Now we obtain two more interpretations of the coefficients of
the toric $g$-contribution polynomial~$g_{n,j}(x)$.

\begin{proposition}
\label{proposition_gamma_contribution} 
Let $B$ be a sparse subset of the set $[n-1]$.
The coefficient of $x^{k}$ in the
toric $g$-contribution
polynomial~$g_{n,|B|}(x)$ is
given by:
\begin{itemize}
\item[(a)]
the number of all $(\emptyset,B)$-compatible Dyck words
of semilength $n$ that contain
exactly $k$ copies of the word $UUD$ as a factor.
\item[(b)]
the number of $123$-avoiding permutations $\pi \in \SSSS_{n}$ such that
$B \subseteq \Asc(\pi)$ and
the inverse permutation $\pi^{-1}$ has exactly $k$ ascents.
\end{itemize}
\end{proposition}
\begin{proof}
Mimicking the proof of Theorem~\ref{theorem_peaks}, let us compute
$g_{n,j}(x+1)$. It suffices to show that the coefficient of $x^{k}$ in
$g_{n,j}(x+1)$ is the number of Dyck words $v$
where we have marked $k$ of the factors~$UUD$.
Label the Dyck word $v$.
Then these marked factors correspond to a sparse set~$A$ of size~$k$ in the set~$[n-1$].
The number of sparse sets of cardinality $k$ is given by the binomial
coefficient~$\binomial{n-k}{k}$.
Hence the sought after number of $(\emptyset,B)$-compatible Dyck words
is $\binomial{n-k}{k} \cdot C_{n-k-j}$,
which is the coefficient of~$x^{k}$ in $g_{n,j}(x+1)$.
The second interpretation follows from Krattenthaler's bijection between
Dyck paths and $123$-avoiding permutations.
\end{proof}

By setting the set $B$ to be empty
in Proposition~\ref{proposition_gamma_contribution}
and switching to the inverse permutation, we obtain
the following corollary.
\begin{corollary}
\label{corollary_g_of_cube}  
The $k$th entry in the toric $g$-vector
of the $n$-dimensional cube is
the number of $123$-avoiding permutations in
the symmetric group~$\SSSS_{n}$
with exactly $k$ ascents.
\end{corollary}

Proposition~\ref{proposition_gamma_contribution} may be also restated in
terms of noncrossing partition statistics, namely in terms of counting
nonsingleton blocks and filler points. 
Recall a partition of $[n]$ is {\em noncrossing} if for every four elements
$1 \leq a < b < c < d \leq n$ if
$a$ and $c$ occur in the same block and $b$ and $d$ occur in the same block,
then $a$, $b$, $c$ and $d$ all lie within the same block.
The notion of filler points was
introduced by Denise and Simion~\cite[Definition~2.5]{Denise_Simion}.
\begin{definition}
Let $\pi$ be a noncrossing partition of the set $[n]$. An element
$2 \leq i \leq n$ is a {\em filler} if one of the following conditions holds:
\begin{enumerate}
\item $i$ is the largest
  element of its block and $i-1$ belongs to the same block as $i$. 
\item  $i$ forms a singleton block and $i-1$ is not the largest
  element of its block. 
\end{enumerate}
\end{definition}  
The set of all filler points of a noncrossing
partition must be sparse: if $i$ is a filler then $i-1$ is not a
filler point. Proposition~\ref{proposition_gamma_contribution}  may be
restated as follows.

\begin{proposition}
\label{proposition_fillers} 
Let $J$ be a sparse subset of the interval $[2,n]$.
Then the coefficient of $x^{k}$ in $g_{n,|J|}(x)$ is the
number of noncrossing partitions of $[n]$ containing exactly $k$ nonsingleton
blocks whose set of filler points contains $J$.
\end{proposition}
\begin{proof}
We use a variant of a well-known bijection between noncrossing
partitions of the set $[n]$ and Dyck paths of semilength $n$;
see for example~\cite[Fig.\ 2 (f)]{Simion}. Given a Dyck path of
semilength~$n$ from $(0,0)$ to $(2n,0)$ using $U$ steps $(1,1)$ and $D$
steps $(1,-1)$, reading from right to left label the down steps from $1$ to
$n$. Treating the down steps as right parentheses and the up steps as
left parentheses, pair each up step with a down step, and transfer
the label on the down step to the matching up step. The labels on the
longest runs of contiguous up steps then form the blocks of a
noncrossing partition. Under this bijection, nonsingleton blocks
correspond to contiguous runs of $U$ steps ending with a factor
$UUD$. Conversely, each factor $UUD$ marks the end of a run of $U$ steps
corresponding to a nonsingleton block. It is left to show that filler
points correspond exactly to the factors $UDD$ under this
correspondence, and the statement will then follow from
Proposition~\ref{proposition_gamma_contribution}.

Consider first a factor $UDD$. Since we labeled the $D$ steps in the right
to left order, adding the labeling yields a factor $UD_{i}D_{i-1}$. A
Dyck word cannot begin with a factor $UDD$, so there is at least one
letter $X$ immediately preceding our factor $UD_{i}D_{i-1}$. If $X=U$
holds, that is, there is a factor $UUD_{i}D_{i-1}$, then the letters $U$
in this factor are matched to the letters $D_{i}$ and $D_{i-1}$: both
$i$ and $i-1$ belong to the same block and $i$ is the maximum element of
this block, hence $i$ is a filler of type (i). If $X=D$
holds, that is, there is a factor $DUD_{i}D_{i-1}$ then the $U$ in this
factor is matched to $D_{i}$ and $\{i\}$ is a singleton block. The
letter $D_{i-1}$ is matched to an earlier $U$ occurring to the left.  This earlier $U$ cannot
be immediately followed by a $D$, hence $i-1$ is not the largest
element of its block. Therefore $i$ is a filler of type (ii). The
converse is straightforward and left to the reader.   
\end{proof}

Substituting $j=0$ into Proposition~\ref{proposition_fillers} yields
a new proof of~\cite[Lemma 6.4]{Hetyei_second_look}.
\begin{corollary}[Hetyei]
\label{corollary_nonsingleton}
The $k$th entry in the toric $g$-vector
of the $n$-dimensional cube is
the number of noncrossing
partitions on the set $[n]$ with exactly $k$ nonsingleton blocks. 
\end{corollary}
It is worth noting that reading Dyck words backwards turns $UUD$
factors into $UDD$ factors and vice versa. Substituting $j=0$ into
the reverse variant of Proposition~\ref{proposition_fillers} yields
the following result of Denise and Simion~\cite{Denise_Simion}, first
pointed out in~\cite[Lemma 6.7]{Hetyei_second_look}. 
\begin{corollary}[Denise--Simion]
\label{corollary_Denise--Simion}
The $k$th entry in the toric $g$-vector
of the $n$-dimensional cube is
the number of noncrossing
partitions on the set $[n]$ with exactly $k$ fillers. 
\end{corollary}
It has been pointed out in~\cite{Hetyei_second_look} that the work of
Denise and Simion~\cite[Remark~3.4]{Denise_Simion} implicitly contains
the description of a simplicial complex whose face numbers are the
coefficients of the powers of $x$ in $g_{n,0}(x)$. The first such
construction was given by Billera, Chan and Liu~\cite{Billera_Chan_Liu}. 
Here we add a third possibility.  It is a direct consequence of
Corollary~\ref{corollary_nonsingleton}.
\begin{corollary}
\label{corollary_nccomplex}
The $k$th entry in the toric $g$-vector
of the $n$-dimensional cube is
the number of
$(k-1)$-dimensional faces in the following simplicial complex:
\begin{enumerate}
\item
The vertices are the subsets of $[n]$ containing at least two elements.
\item A collection $\{S_{1},S_{2},\ldots,S_{k}\}$ is a face if and only if there is a
noncrossing partition on the set~$[n]$ whose nonsingleton blocks are exactly
the sets $S_{1},S_{2},\ldots,S_{k}$. 
\end{enumerate}  
\end{corollary}

\section{Generating functions of the toric $g$-contributions}
\label{section_Generating_functions}

In this section we compute the ordinary generating function of
the toric $g$-contribution polynomials~$g_{n,j}(x)$
defined in equation~\eqref{equation_gp}.
Note this is a valid definition for all
$0 \leq j \leq n$ even if we use only the polynomials $g_{n,j}(x)$
satisfying $0\leq j\leq \half{n}$ in our toric $g$-polynomial
formulas. We extend the definition of $g_{n,j}(x)$ to all $j\geq 0$ by
setting $g_{n,j}(x)=0$ whenever $j>n$. Keeping this extension in mind,
we obtain the following recurrence.
\begin{lemma}
\label{lemma_grecurrence}  
The toric $g$-contribution polynomials~$g_{n,j}(x)$ satisfy 
the recurrence
\begin{equation}
g_{n,j}(x)= g_{n-1,j-1}(x)+(x-1)\cdot g_{n-2,j-1}(x)
\end{equation}
for $n \geq 2$ and $j \geq 1$.
\end{lemma}  
\begin{proof}
We prove the equivalent expression
\begin{align}
g_{n,j}(x+1) &= g_{n-1,j-1}(x+1)+x \cdot g_{n-2,j-1}(x+1)
\label{equation_(5.1)_shifted}
\end{align}
for $n \geq 2$ and $j \geq 1$.
The coefficient of $x^{k}$ is 
$C_{n-j-k}\binom{n-k}{k}$ on the left-hand side 
of~\eqref{equation_(5.1)_shifted} and
$C_{n-j-k}\binom{n-1-k}{k}+C_{n-j-k}\binom{n-1-k}{k-1}$ on the right-hand
side. The statement is a direct consequence of 
the Pascal recursion
$\binom{n-k}{k} = \binom{n-1-k}{k} + \binom{n-1-k}{k-1}$.
\end{proof}

We introduce the generating function $G_{j}(x,t)$
of the toric $g$-contribution polynomials~$g_{n,j}(x)$.
\begin{align}
\label{equation_G_j}
G_{j}(x,t) & = \sum_{n \geq 2j} g_{n,j}(x)\cdot t^{n} .
\end{align}
Lemma~\ref{lemma_grecurrence} implies
\begin{align*}
G_{j}(x,t)
& =
(t+(x-1)\cdot t^{2}) \cdot G_{j-1}(x,t)
=
t \cdot (1-t+xt) \cdot G_{j-1}(x,t)
\end{align*}  
and hence
\begin{align*}
G_{j}(x,t)
& =
t^{j} \cdot (1-t+xt)^{j} \cdot G_{0}(x,t) .
\end{align*}

The table of the coefficients of $G_{0}(x,t)$ is sequence A091156
in~\cite{OEIS}. It is stated in~\cite{OEIS} that the generating function
$G_{0}(x,t)$ is the solution of the quadratic equation
\begin{equation}
\label{equation_gzero}  
G_{0}(x,t)= t\cdot (1-t+xt)\cdot G_{0}(x,t)^{2} + 1.
\end{equation}
This statement is equivalent to the following recurrence.
\begin{proposition}
The toric $g$-contribution polynomials $g_{n,0}(x)$ satisfy
\begin{align*}
g_{n,0}(x)
& =
(x-1) \cdot \sum_{m=2}^{n-1} g_{m-2,0}(x)\cdot g_{n-m,0}(x)
+
\sum_{m=1}^{n-1} g_{m-1,0}(x)\cdot g_{n-m,0}(x)
\quad \text{for $n \geq 1$.}
\end{align*}
\end{proposition}
\begin{proof}
Recall that the coefficient of $x^{k}$ in
$g_{n,0}(x)$ is the number of Dyck words of semilength $n$ containing
exactly $k$ factors $UUD$. Let $m$ be the least positive integer such
that the associated Dyck path touches the horizontal axis at
$(2m,0)$. In this case the Dyck word factors as $UuDv$ where $u$ is a
Dyck word of semilength $m-1$ and $v$ is a Dyck word of semilength
$n-m$. The words $u$ and $v$ may be chosen independently. As $v$ ranges
over all possible choices, we obtain a factor of $g_{n-m,0}(x)$. 

\noindent 
Case 1: The word $u$ begins with a factor $UD$, hence we
have the word $UUDu'Dv$. In this case the first factor $UUD$ contributes a
factor of $x$, and $u'$ is a Dyck word of semilength $m-2$, which may be
chosen independently. The contribution of all Dyck words belonging to
this case is $x \cdot g_{m-2,0}(x) \cdot g_{n-m,0}(x)$
(where $m \geq 2$ must hold).

\noindent 
Case 2: The word $u$ does not begin with a factor
$UD$. If we choose the word $u$ in all possible ways, we obtain a
contribution of $g_{m-1,0}(x) \cdot g_{n-m,0}(x)$. From this we must subtract
the contribution of the words~$u$ beginning with a factor $UD$. Based on
the previous case, $g_{m-2,0}(x) \cdot g_{n-m,0}(x)$ must be subtracted
(when $m \geq 2$ holds).
\end{proof}
The solution of equation~\eqref{equation_gzero} is 
\begin{equation}
G_{0}(x,t)=\frac{1 - \sqrt{4(1-x)t^{2} - 4t + 1}}{2t(1 - t + xt)}.
\end{equation}   
This expression, combined with~\eqref{equation_G_j} yields
\begin{equation}
G_{1}(x,t)=\frac{1 - \sqrt{4(1-x)t^{2} - 4t + 1}}{2}.
\end{equation}   
The table of the coefficients of $G_{1}(x,t)$ is sequence A091894 
in~\cite{OEIS}. Among other interpretations listed there, the entry
tabulates the number of Dyck words of
semilength $n$ containing $k$ factors $DDU$. It should be noted that the
generating function given there (using our variable assignment) is
$$
\frac{1 - \sqrt{4(1-x)t^{2} - 4t + 1}}{2xt},
$$
which is $G_{1}(x,t)/(xt)$ in our notation.

Another way to compute
the toric $g$-contribution polynomials~$g_{n,j}(x)$
is by introducing the {\em peak polynomials}. 

\begin{definition}
Given a pair $(n,m)$ of nonnegative integers satisfying $0\leq m\leq
2n$, the {\em peak polynomial $p_{n,m}(x)$} is the total weight
of all Dyck words of semilength $n$, where the weight of each Dyck
word is $x$ to the power of the number of $UD$ factors occurring in the
length $m$ prefix of the word.
\end{definition}  

Lemma~\ref{lemma_peaks} and Theorem~\ref{theorem_peaks} imply
\begin{align}
g_{n,j}(x) & = p_{n-j,n}(x).
\end{align}  

Let $N_{k}(x)$ denote the Narayana polynomial
\begin{align*}
N_{k}(x)
& =
\frac{1}{k} \cdot
\sum_{j=1}^{k} \binom{k}{j} \cdot \binom{k}{j - 1} \cdot x^{j} ,
\end{align*}
for $k>0$ and set $N_{0}(x)=x$.
\begin{proposition}
\label{prop_peakrec}  
The peak polynomial $p_{n,m}(x)$ satisfies the recurrence
\begin{align*}
p_{n,m}(x)
& =
\sum_{k=0}^{\lfloor\half{(m-2)}\rfloor} N_{k}(x)\cdot p_{n-k-1,m-2k-2}(x)
+
\sum_{k=\lceil\half{(m-1)}\rceil}^{n-1} p_{k,m-1}(x)\cdot C_{n-k-1} 
\end{align*}
for $1\leq m\leq 2n$, and from the initial conditions $p_{n,0}(x)=C_{n}$ for $n \geq 0$.
\end{proposition}  
\begin{proof}
The stated initial conditions are direct: we do not count peaks when
$m=0$, and the number of all Dyck words of semilength $n$
is the Catalan number $C_{n}$. From now on we may assume that $m \geq 1$
holds and we consider a Dyck word $w$ of semilength $n \geq \half{m}$. Let
$k+1>0$ be the semilength of the shortest nonempty prefix $u$ in which
the number of letters $U$ equals the number of letters $D$. The word $u$
is then of the form $Uu'D$ where $u'$ is a Dyck word of semilength
$k$. Furthermore $w$ factors as $w=uv$ where $v$ is also a Dyck word.
We distinguish two cases.

\noindent 
Case 1: The inequality $2k+1\leq m-1$ holds.
Equivalently, we have
$k\leq \lfloor\half{(m-2)}\rfloor$. Note that this case does not occur
when $m=1$. In this case all peaks belonging to
the word $u'$ contribute a factor of $x$. As $u'$ varies over all Dyck
words of semilength $k$, their total contribution is the Narayana
polynomial $N_{k}(x)$. Note that the observation also holds when $k=0$ and
$u=UD$ hold, because we fixed the value of $N_{0}(x)$ to be $x$. The
second factor $v$ is an independently selected Dyck word of semilength
$n-k-1$ where we count the peaks among the first $m-2k-2$ letters. They
contribute a factor of $p_{n-k-1,m-2k-2}(x)$. 

\noindent 
Case 2: The inequality $2k+1\geq m$, that is,
$k\geq \lceil\half{(m-1)}\rceil$ holds. In this case all peaks contributing a
factor of $x$ belong to the factor $u'$. As $u'$ varies over all Dyck
words of semilength $k$, their total contribution is $p_{k,m-1}(x)$. 
The second factor $v$ is an independently selected Dyck word of semilength
$n-k-1$ whose peaks are not counted. They
contribute a factor of $C_{n-k-1}$.
\end{proof}
  
The generating function
$N(x,t)=\sum_{n \geq 0}N_{n}(x) t^{n}$ may be obtained from the well-known
generating function formula of the Narayana polynomials~\cite[Eq.~(9)]{Simion}
by adding $x-1$ to its constant term. We have
\begin{align}
N(x,t)
& =
\frac{1+t-xt-\sqrt{1-2(1+x)t+(1-x)^{2}t^{2}}}{2t}+x-1 
\nonumber \\
& =
\frac{1+xt-t-\sqrt{1-2(1+x)t+(1-x)^{2}t^{2}}}{2t}.
\end{align}  
Substituting $x=1$ into $N_{n}(x)$ yields the Catalan number $C_{n}$
and hence $N(1,t) = C(t)$.

We introduce the generating function $P(x,y,z)$
for the peak polynomials $p_{n,m}(x)$
\begin{align*}
P(x,y,z)
& =
\sum_{m \geq 0}\sum_{n \geq \lceil\half{m}\rceil} p_{n,m}(x) \cdot y^{n}z^{m} .
\end{align*}
Proposition~\ref{prop_peakrec} implies the following identity.
\begin{theorem}
The generating function $P(x,y,z)$ is given by
\begin{align*}
P(x,y,z)
& =
\frac{C(y)}{1 - yz^{2} \cdot N(x,yz^{2}) - yz \cdot C(y)}. 
\end{align*}
\label{theorem_peakgen} 
\end{theorem}
\begin{proof}
Assume that  $m \geq 1$ and $n \geq \lceil\half{m}\rceil$ hold.
By replacing $k$ with $n-1-k$ in the second sum of the recurrence stated
in Proposition~\ref{prop_peakrec}, we have the following equivalent
recurrence: 
\begin{align*}
p_{n,m}(x)
& =
\sum_{k=0}^{\lfloor\half{(m-2)}\rfloor} N_{k}(x)\cdot p_{n-k-1,m-2k-2}(x)
+
\sum_{k=0}^{n-1-\lceil\half{(m-1)}\rceil} p_{n-k-1,m-1}(x)\cdot C_{k}. 
\end{align*}
Multiplying both sides with $y^{n}z^{m}$, we obtain
\begin{align*}
y^{n}z^{m} \cdot p_{n,m}(x)
& =
\sum_{k=0}^{\lfloor\half{(m-2)}\rfloor}
N_{k}(x) \cdot (yz^{2})^{k+1}\cdot y^{n-k-1}z^{m-2k-2} \cdot p_{n-k-1,m-2k-2}(x)\\
& +
\sum_{k=0}^{n-1-\lceil\half{(m-1)}\rceil}
C_{k} \cdot y^{k}\cdot yz\cdot y^{n-k-1}z^{m-1} \cdot p_{n-k-1,m-1}(x). 
\end{align*}
Summing over all pairs $(n,m)$ satisfying $m \geq 1$ and $n \geq
\lceil\half{m}\rceil$ we obtain all terms of $P(x,y,z)$ except for
those with $m=0$. The sum of these terms is $\sum_{n \geq 0} C_{n}
y^{n}=C(y)$. Hence we obtain 
\begin{align*}
P(x,y,z)
& =
C(y) + \left(yz^{2} \cdot N(x,yz^{2})+ yz \cdot C(y)\right) \cdot P(x,y,z). 
\end{align*}
Solving this equation for $P(x,y,z)$ yields the expression
stated in the theorem. 
\end{proof}

As a consequence of Theorem~\ref{theorem_peakgen}
and equation~\eqref{equation_gp} we may write
$G_{j}(x,t)$ as
\begin{align*}
G_{j}(x,t)
& =
\sum_{n \geq 0} p_{n-j,n}(x) \cdot t^{n}
=
\sum_{n \geq 0} [y^{n-j}z^{n}] P(x,y,z)\cdot t^{n}
=
\sum_{n \geq 0} [y^{n}z^{n}] y^{j} \cdot P(x,y,z)\cdot t^{n}.
\end{align*}
In particular, substituting $j=0$ yields
$G_{0}(x,t)=\sum_{n \geq 0}  [y^{n}z^{n}]  P(x,y,z)\cdot t^{n}$.
Here we use the notation~$[\cdot]$ to denote extracting the coefficient
from the expression that follows.

\section{The toric $g$-polynomial of the associahedron}
\label{section_associahedron}

We now turn our attention to the associahedron.
The $n$-dimensional associahedron $\Ass_{n}$
is a well-studied simple polytope
where the number of vertices is given by
the Catalan number $C_{n+1}$.
This polytope was first discovered by Tamari~\cite{Tamari} and
rediscovered by Stasheff~\cite{Stasheff}.
There are many ways to realize this polytope.
For a survey we refer to the article by
Ceballos, Santos and Ziegler~\cite{Ceballos_Santos_Ziegler}.

Postnikov, Reiner and Williams computed the 
$\gamma$-vector of the $n$-dimensional
associahedron~\cite[Proposition~11.14]{Postnikov_Reiner_Williams}.
\begin{proposition}[Postnikov--Reiner--Williams]
The $j$th entry of the $\gamma$-vector
of the $n$-dimensional associahedron $\Ass_{n}$ given by
\begin{align*}
\gamma_{j}(\Ass_{n}) & = C_{j} \cdot \binom{n}{2j} .
\end{align*}
\label{proposition_Postnikov--Reiner--Williams_associahedron}
\end{proposition}

Hence the generating function of the $\gamma$-polynomial
of the $n$-dimensional associahedron is
\begin{align*}
\sum_{n \geq 0} \gamma(\Ass_{n},x) \cdot t^{n}
& =
\sum_{n \geq 0} \sum_{j=0}^{\left\lfloor\half{n}\right\rfloor}
C_{j} \cdot \binom{n}{2j} \cdot x^{j} \cdot t^{n}
=
\sum_{j \geq 0} C_{j} \cdot x^{j} \cdot
\sum_{n \geq 2j} \binom{n}{2j} \cdot t^{n} \\
&=
\sum_{j \geq 0} C_{j} \cdot \frac{x^{j}t^{2j}}{(1-t)^{2j+1}}
=
\frac{1}{(1-t)} \cdot C\left(\frac{xt^{2}}{(1-t)^{2}}\right) ,
\end{align*}
where $C(u)$ is the generating function of the Catalan numbers.
Using expression~\eqref{equation_Catalan_generating_function},
the above equation simplifies to
\begin{align*}
\sum_{n \geq 0} \gamma(\Ass_{n},x) \cdot t^{n}
& =
\frac{1 - t - \sqrt{(1 - t)^{2} - 4xt^{2}}}{2xt^{2}}.  
\end{align*}

\begin{table}[t]
\begin{tabular}{|c|| r| r| r| r| r|}  
\hline
  \backslashbox{$n$}{$j$} & 0 & 1 & 2 & 3 & 4 \\
\hline  
  1& 1 &&&&\\
  2& 1 &       2 &&&\\
  3& 1 &     10 &&&\\
  4& 1 &     37 &       10 &&\\
  5& 1 &   126 &     105 &&\\
  6& 1 &   422 &     714 &       70 & \\ 
  7& 1 & 1422 &   4032 &   1176 & \\
  8& 1 & 4853 & 20628 & 11928 & 588 \\
\hline
\end{tabular}
\vspace*{1 mm}
\caption{The toric $g$-vector of the associahedron up to dimension $8$.} 
\label{table_g-polynomial_associahedron}  
\end{table}

The main result of this section
is that the
toric $g$-polynomial of the associahedron has a combinatorial
interpretation in terms of the ascent statistics of $123$-avoiding
parking functions.
As reported in~\cite{Adeniran_Pudwell}, the number of all $123$-avoiding
parking functions was first computed by Qiu and
Remmel~\cite{Qiu_Remmel}. We refine Qiu's
formulas~\cite{Qiu_slides} to the following result.

\begin{figure}[t]
\begin{center}
\begin{tikzpicture}[scale = 0.5]
\draw[step=1,gray,very thin] (0,0) grid (10,10);
\draw[thick,-] (0,0) -- (4,0) -- (4,2) -- (6,2) -- (6,5) -- (9,5) -- (9,7) -- (10,7) -- (10,10)
-- (9,10) -- (9,9) -- (6,9) -- (6,7) -- (5,7) -- (5,6) -- (4,6) -- (4,5) -- (3,5) -- (3,4) -- (2,4)
-- (2,2) -- (0,2) -- cycle;
\draw[dashed] (0,0) -- (10,10);
%%%%
\node at (0.5,0.5) {7};
\node at (0.5,1.5) {10};
\node at (2.5,2.5) {5};
\node at (2.5,3.5) {9};
\node at (3.5,4.5) {8};
\node at (4.5,5.5) {2};
\node at (5.5,6.5) {6};
\node at (6.5,7.5) {1};
\node at (6.5,8.5) {4};
\node at (9.5,9.5) {3};
%%%% "Krattenthaler"
\node at (3.5,0.5) {7};
\node at (3.5,1.5) {10};
\node at (5.5,2.5) {5};
\node at (5.5,3.5) {9};
\node at (5.5,4.5) {8};
\node at (8.5,5.5) {2};
\node at (8.5,6.5) {6};
\node at (9.5,7.5) {1};
\node at (9.5,8.5) {4};
\node at (9.5,9.5) {3};
\draw (3.5,0.5) circle (0.45);
\draw (3.5,1.5) circle (0.45);
\draw (5.5,2.5) circle (0.45);
\draw (5.5,3.5) circle (0.45);
\draw (5.5,4.5) circle (0.45);
\draw (8.5,5.5) circle (0.45);
\draw (8.5,6.5) circle (0.45);
\draw (9.5,7.5) circle (0.45);
\draw (9.5,8.5) circle (0.45);
\draw (9.5,9.5) circle (0.45);

\node[rotate=90] at (-0.6,5) {Garsia--Haiman};
\node[rotate=-90] at (10.6,5) {Krattenthaler};
\end{tikzpicture}
\end{center}
\caption{The pair of Dyck paths representing the parking function
$(7,5,10,7,3,6,1,4,3,1)$.}
\label{figure_pair_of_paths}
\end{figure}
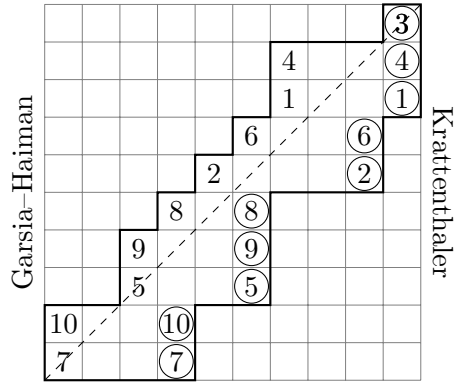

\begin{theorem}
\label{theorem_Margie}  
The coefficient of $x^{k}$ in the toric $g$-polynomial of the
$n$-dimensional associahedron~$\Ass_{n}$
is the number of $123$-avoiding parking
functions on the set $[n]$  having exactly $k$ ascents.  
\end{theorem}
\begin{proof}
The key tools are the Garsia--Haiman and the Krattenthaler bijections
presented in Section~\ref{section_preliminaries}.
Observe that we have combined the two pictures
in Figures~\ref{figure_Krattenthaler}
and~\ref{figure_Garsia_Haiman}
into Figure~\ref{figure_pair_of_paths}.
We begin by characterizing $123$-avoiding parking functions in terms of their
Garsia--Haiman representation. 

Given a parking function $f : [n] \longrightarrow [n]$,
the Garsia--Haiman bijection yields the pair $(\pi,v)$.
It is almost an immediate consequence of the
construction and the definitions that the function $f$ is $123$-avoiding
if and only if the permutation $\pi$ avoids the pattern $123$:
there is no triple $1 \leq i_{1} < i_{2} < i_{3} \leq n$ satisfying
$\pi(i_{1}) < \pi(i_{2}) < \pi(i_{3})$.
Indeed, the existence of such a
triple $(i_{1},i_{2},i_{3})$ is equivalent to stating that the horizontal
coordinates of the labels $\pi(i_{1})$, $\pi(i_{2})$ and $\pi(i_{3})$
are weakly increasing which is the same as having
$f(\pi(i_{1})) \leq f(\pi(i_{2})) \leq f(\pi(i_{3}))$.
Before moving on to the discussion
of the other Dyck path in Figure~\ref{figure_pair_of_paths}, we note
that the permutation~$\pi$ being $123$-avoiding precludes the
presence of the factor $UUU$ in the Dyck word $v$.

Since the permutation $\pi$ is $123$-avoiding,
we apply the Krattenthaler bijection to obtain the
Dyck path below $y=x$ encoded by the Dyck word $w$.
Hence this gives the pair of Dyck paths $(v,w)$.

Not any $UUU$-avoiding Dyck path above the line $y=x$ can be paired with
any Dyck path below it due to the following requirement: whenever
$f(\pi(i))=f(\pi(i+1))$ holds, we must have
$\pi(i)<\pi(i+1)$. This follows because in the Garsia--Haiman encoding of a
parking function the labels with the same horizontal coordinate must
increase in the vertical order. After observing that $\pi(i)$ labels
the $i$th letter $U$ in the Dyck word $v$,
we obtain the following {\em compatibility criterion}: if
$U_{i}$ and $U_{i+1}$ are consecutive letters of the Dyck word $v$
then $D_{i}$ and $D_{i+1}$ must be
consecutive letters in the Dyck word~$w$
immediately preceded by a letter $U$. In other words, if
$U_{i}U_{i+1}$ is a factor in the Dyck word $v$, then $UD_{i}D_{i+1}$
must be a factor in the Dyck word $w$. Indeed, the index $i$ is an ascent of a
$123$-avoiding permutation $\pi$ if and only if $\pi(i)$ is a
left-to-right minimum of $\pi$ and $\pi(i+1)$ is not. Conversely, if
the compatibility criterion is satisfied, a permutation $\pi$ encoded by
the Dyck word $w$ provides a valid labeling for the
compatible Dyck word $v$, yielding the Garsia--Haiman
encoding of a $123$-avoiding parking function.

Fix $j$ as the number of factors $U_{i}U_{i+1}$ in
the Dyck word $v$.
By Lemma~\ref{lemma_j_UU} part~(a)
we know that the number of such
Dyck words is $C_{j} \cdot \binom{n}{2j}$.
This is also the entry $\gamma_{j}$ in the $\gamma$-vector of the
associahedron.
To complete the proof using
Proposition~\ref{proposition_gamma_contribution} we only need to make
the following observation: $i$ is an ascent of the parking function~$f$
if and only if it is an ascent of the {\em inverse} of the permutation~$\pi$.
Indeed $f(i)\leq f(i+1)$ is equivalent to stating that $i$
precedes $i+1$ in the permutation~$\pi$ which is equivalent to
$\pi^{-1}(i)<\pi^{-1}(i+1)$. In terms of the Dyck word $w$, the
label $i$ must precede the label $i+1$ on a $D$~step: this is only
possible if $i$ is a left-to-right minimum in~$\pi$ and the
left-to-right minimum immediately preceding $i$ is not
$i+1$. Equivalently, the $UD$ factor whose $D$ step is labeled $i$ is
immediately preceded by a letter~$U$. The converse is also true, hence
the number of ascents of $f$ is the same as the number of factors~$UUD$
in the Dyck path encoding~$\pi$.     

The statement now follows from
Proposition~\ref{proposition_gamma_contribution}. 
\end{proof}

We conclude this section by computing the leading coefficients of the
toric $g$-polynomials. 

\begin{proposition}
(a)
The leading coefficient of the toric $g$-polynomial of
the $2m$-dimensional associahedron is given by
\begin{align*}
[x^{m}] g(\Ass_{2m},x)
& =
C_{m} \cdot C_{m+1} .
\end{align*}
This enumerates
the number of lattice paths in the first quadrant $\Nnn^{2}$
starting and ending at the origin
having length $2m$ and taking the steps
$(\pm 1, 0)$ and $(0, \pm 1)$. \\
(b)
The leading coefficient of the toric $g$-polynomial of
the $(2m+1)$-dimensional associahedron is given by
\begin{align*}
[x^{m}] g(\Ass_{2m+1},x)
& =
C_{m+1} \cdot \binom{2m+3}{m} .
\end{align*}
This quantity is
the number of lattice paths in the half plane $\Nnn\times \Zzz$
starting and ending at the origin
having length $2m+2$ and taking the steps
$(\pm 1, 0)$ and $(0, \pm 1)$ in such a way that the lattice path
does not stay inside the first quadrant $\Nnn^{2}$.
\label{corollary_associahedron_2m_m}
\end{proposition}
\begin{proof}
To show part (a),
by Corollary~\ref{corollary_2m_m}
and
Proposition~\ref{proposition_Postnikov--Reiner--Williams_associahedron}
the leading coefficient is given by
\begin{align*}
[x^{m}] g(\Ass_{2m},x)
& =
\sum_{j=0}^{m} \binom{2m}{2j} \cdot C_{j} \cdot C_{m-j}.
\end{align*}
Observe that the $j$th term is the number of
lattice paths described in statement (a) taking
$2j$ horizontal steps and $2m-2j$ vertical steps. This expression is
equal to $C_{m} \cdot C_{m+1}$ by \cite[Eq.~(I)]{Cori_Dulucq_Viennot}.

For part (b), we similarly
obtain that the leading coefficient of the toric $g$-polynomial
of the $(2m+1)$-dimensional associahedron is
\begin{align*}
[x^{m}] g(\Ass_{2m+1},x)
& =
(m+1) \cdot \sum_{j=0}^{m} \binom{2m+1}{2j} \cdot
C_{j} \cdot C_{m+1-j} \\
& =
\sum_{j=0}^{m} \binom{2m+2}{2j} \cdot C_{j} \cdot
(m+1-j)\cdot C_{m+1-j} \\
& =
\sum_{j=0}^{m} \binom{2m+2}{2j} \cdot C_{j} \cdot
  \left(\binom{2m+2-2j}{m+1-j}-C_{m+1-j}\right),
\end{align*}
where we used
$a \cdot \binom{a-1}{b} = \binom{a}{b} \cdot (a-b)$
and
$k \cdot C_{k} = \binom{2k}{k-1} = \binom{2k}{k} - C_{k}$.
Observe that the $j$th term is the number of
lattice paths described in statement (b) taking
$2j$ horizontal steps and $2m+2-2j$ vertical
steps. 
The factor $C_{j}$ guarantees that the path
stays inside the half plane $\Nnn \times \Zzz$. However,
the next factor $\binom{2m+2-2j}{m+1-j}-C_{m+1-j}$
shows that the path leaves the quadrant $\Nnn^{2}$.
By~\cite[Eq.~(1.4)]{Guy_Krattenthaler_Sagan} the number of all
lattice paths in $\Nnn \times \Zzz$ of length $2m+2$ from $(0,0)$ to $(0,0)$  
equals
\begin{align}
\binom{2m+2}{m+1}^{2} - \binom{2m+2}{m}^{2}
& =
\left( \binom{2m+2}{m+1} - \binom{2m+2}{m} \right)
\cdot
\left( \binom{2m+2}{m+1} + \binom{2m+2}{m} \right)
\nonumber \\
& =
C_{m+1} \cdot \binom{2m+3}{m+1} .
\label{equation_N_Z}
\end{align}
From this expression we need to subtract 
the number of lattice paths
that stay inside $\Nnn\times \Nnn$,
that is, $C_{m+1} \cdot C_{m+2}$ by part (a).
It is straightforward to see that the
difference is $C_{m+1} \cdot \binom{2m+3}{m}$. 
\end{proof}

\section{The toric $g$-polynomial of the cyclohedron}
\label{section_cyclohedron}

\begin{table}[t]
\begin{tabular}{|c|| r| r| r| r| r|}  
\hline
  \backslashbox{$n$}{$j$} & 0 & 1 & 2 & 3 & 4 \\
\hline  
  1& 1 &&&&\\
  2& 1 &        3 &&&\\
  3& 1 &      16 &&&\\
  4& 1 &      65 &         20 &&\\
  5& 1 &    246 &       225 &&\\
  6& 1 &    917 &     1659 &      175 & \\ 
  7& 1 &  3424 &   10192 &    3136 & \\
  8& 1 & 12861 &  56664 &  34104 & 1764 \\
\hline
\end{tabular}
\vspace*{1 mm}
\caption{The toric $g$-vector of the cyclohedron up to dimension $8$.} 
\label{table_g-polynomial_cyclohedron}  
\end{table}

The $n$-dimensional cyclohedron~$\Cyc_{n}$
is a type $B$ version of the associahedron.
The cyclohedron was constructed as a combinatorial object by
Bott and Taubes~\cite{Bott_Taubes},
while
polytopal realizations were given by
Markl~\cite{Markl},
Simion~\cite{Simion_type_B}
and
the authors~\cite{Ehrenborg_Hetyei_Readdy_pulling}.
Postnikov, Reiner and Williams computed the $\gamma$-vector
of the $n$-dimensional
cyclohedron~\cite[Proposition~11.15]{Postnikov_Reiner_Williams}.

In this section we give a combinatorial interpretation
of the toric $g$-vector of the cyclohedron
in terms of $123$-avoiding {\em functions}.
For the cyclohedron
the computation has a similar flavor
as that of the associahedron.

\begin{proposition}[Postnikov--Reiner--Williams]
The $j$th entry of the $\gamma$-vector
of the $n$-dimensional cyclohedron~$\Cyc_{n}$ given by
\begin{align*}
\gamma_{j}(\Cyc_{n}) & = \binom{2j}{j} \cdot \binom{n}{2j} .
\end{align*}
\label{proposition_Postnikov--Reiner--Williams_cyclohedron}
\end{proposition}

\begin{lemma}
The generating function 
for the $\gamma$-polynomials of the cyclohedra is given by
\begin{align*}
\sum_{n \geq 0} \gamma(\Cyc_{n},x) \cdot t^{n}
&=
\frac{1}{\sqrt{(1-t)^{2}-4xt^{2}}}.
\end{align*}
\end{lemma}
\begin{proof}
The computation is as follows:
\begin{align*}
\sum_{n \geq 0} \gamma(\Cyc_{n},x) \cdot t^{n}
&=
\sum_{n \geq 0}
\sum_{j=0}^{\left\lfloor\half{n}\right\rfloor}
\binom{2j}{j} \cdot \binom{n}{2j} \cdot x^{j} \cdot t^{n}
=
\sum_{j \geq 0} \binom{2j}{j} \cdot x^{j} \cdot
\sum_{n \geq 2j} \binom{n}{2j} \cdot t^{n} \\
& =
\sum_{j \geq 0}
\binom{2j}{j} \cdot \frac{x^{j}t^{2j}}{(1-t)^{2j+1}}
=
\frac{1}{(1-t) \cdot \sqrt{1-4\frac{xt^{2}}{(1-t)^{2}}}} \\
& =
\frac{1}{\sqrt{(1-t)^{2}-4xt^{2}}}.
\qedhere
\end{align*}
\end{proof}
We have the
following combinatorial interpretation for the toric
$g$-vector of the cyclohedron.
\begin{theorem}
\label{theorem_Margie2}  
The coefficient of $x^{k}$ in the toric $g$-polynomial of the
$n$-dimensional cyclohedron~$\Cyc_{n}$
is the number of $123$-avoiding functions
$f:[n]\longrightarrow [n]$ having exactly $k$ ascents.  
\end{theorem}
\begin{proof}
The proof idea is analogous to the argument for
Theorem~\ref{theorem_Margie}. The only difference is that the
Garsia--Haiman lattice
path does not need to be a Dyck path.
Instead it is a lattice path described in Lemma~\ref{lemma_j_UU} part~(b).
This lemma states that the number of such paths
is $\binom{n}{j} \cdot \binom{2j}{j}$,
which is the $\gamma$-vector of the cyclohedron.
\end{proof}

We conclude this section by determining the leading coefficients of the
toric $g$-polynomials. 
\begin{proposition}
(a)
The leading  coefficient of the toric $g$-polynomial of
the $2m$-dimensional cyclohedron is 
given by
\begin{align*}
[x^{m}] g(\Cyc_{2m},x)
& = 
C_{m} \cdot \binom{2m+1}{m} .
\end{align*}
This quantity enumerates 
the number of lattice paths in the half plane $\Zzz \times \Nnn$
starting and ending at the origin
having length $2m$ and taking the steps
$(\pm 1, 0)$ and $(0, \pm 1)$. \\
(b)
The leading  coefficient of the toric $g$-polynomial of
the $(2m+1)$-dimensional cyclohedron is 
given by
\begin{align*}
[x^{m}] g(\Cyc_{2m+1},x)
& = 
\binom{2m+2}{m}^{2} .
\end{align*}
This quantity is the number of
lattice paths in the plane $\Zzz \times \Zzz$
starting and ending at the origin
having length $2m+2$ and taking the steps
$(\pm 1, 0)$ and $(0, \pm 1)$ such that the lattice path does not stay
inside the half plane $\Zzz \times \Nnn$. 
\label{corollary_cyclohedron_2m_m}
\end{proposition}
\begin{proof}
To prove part~(a) Corollary~\ref{corollary_2m_m}
and
Proposition~\ref{proposition_Postnikov--Reiner--Williams_cyclohedron}
implies that the leading coefficient 
is given by
\begin{align*}
[x^{m}] g(\Cyc_{2m},x)
& =
\sum_{j=0}^{m} \binom{2m}{2j} \cdot \binom{2j}{j} \cdot C_{m-j} .
\end{align*}
Observe that the $j$th term is the number of
lattice paths described in statement (a) taking
$2j$ horizontal steps.
The expression $C_{m} \cdot \binom{2m+1}{m}$ follows
from equation~\eqref{equation_N_Z}.

Similarly for part~(b) we have
\begin{align*}
[x^{m}] g(\Cyc_{2m+1},x)
& =  
(m+1) \cdot 
\sum_{j=0}^{m} \binom{2m+1}{2j} \cdot \binom{2j}{j} \cdot C_{m+1-j}\\
&=\sum_{j=0}^{m}\binom{2m+2}{2j} \cdot \binom{2j}{j}\cdot (m+1-j) \cdot
C_{m+1-j}\\ 
&=\sum_{j=0}^{m+1}\binom{2m+2}{2j} \cdot \binom{2j}{j}\cdot
\left(\binom{2m+2-2j}{m+1-j}-C_{m+1-j}\right) .
\end{align*}
Observe that the $j$th term is the number of
lattice paths described in statement (b) taking
$2j$ horizontal steps. By \cite[Eq.~(1.2)]{Guy_Krattenthaler_Sagan}, the
number of all lattice paths of length $2m+2$ from $(0,0)$ to $(0,0)$ is
$\binom{2m+2}{m+1}^{2}$. From these we need to subtract the 
number of lattice paths which stay inside the half plane $\Zzz \times \Nnn$.
By part~(a) and equation~\eqref{equation_N_Z} this number is
\begin{align*}
C_{m+1} \cdot \binom{2m+3}{m+1}
& =
\left(\binom{2m+2}{m+1}-\binom{2m+2}{m}\right)\cdot
\left(\binom{2m+2}{m+1}+\binom{2m+2}{m}\right). 
\end{align*}
The difference is the square $\binom{2m+2}{m}^{2}$. 
\end{proof}

\section{Parking trees}
\label{section_parking_trees}

In this section we introduce {\em parking
trees} and characterize those that represent $123$-avoiding parking
functions. Our presentation of a parking tree is based upon an exercise due to
Sagan~\cite[Ch~1.\ Exercise~(32)(c)]{Sagan} and is a slight
generalization of the constructions described by Yan~\cite[Section
  13.2.3]{Yan_chapter}. The first bijection between labeled trees and
parking functions is due to Kreweras~\cite{Kreweras_suites},
whereas Knuth found a simpler construction~\cite[Section~6.4, Exercise
  31, solution  (c)]{Knuth}; see also~\cite[Section~13.2.2]{Yan_chapter}. These
parking trees are {\em different} from ours; see
Section~\ref{section_Concluding_remarks} for further details.

\begin{definition}
A {\em parking tree} is a bilabeled rooted plane tree on $n+1$ vertices
whose vertices, respectively edges, are bijectively labeled with the
elements of the set $[n+1]$, respectively $[n]$, subject to the
following conditions: 
\begin{enumerate}
\item The labeling on the vertices is {\em increasing}: the label of
  each vertex is less than the label of any of its children.
\item If $i$ and $j$ satisfying  $i<j$ are labels of edges connecting
  the same parent vertex to different children
then the edge labeled $i$ lies to the {\em left} of the edge
labeled $j$. 
\end{enumerate}  
\end{definition}

Using Exercise~1.63 in~\cite{Stanley_EC1}
we note that there are $n!^{2}$ parking trees
on $n+1$ nodes.

Each parking tree may be used to encode a parking function using the
fact that every edge in a rooted plane tree connects a parent vertex to
its child. 

\begin{definition}
\label{definition_parking}  
Given a parking tree on $n+1$ vertices, associate a parking function
$f:[n] \longrightarrow [n]$ to it as follows.  For each $i \in [n]$ set $f(i)$ to
be the label of the parent vertex incident to the edge labeled $i$. 
\end{definition}

\begin{figure}[t]
\begin{center}
\tikzstyle{place}=[circle,draw=black,thick,
                   inner sep=0pt,minimum size=4.5mm]
\begin{tikzpicture}[scale = 0.8]
\node (one) at (0,6) [place] {\small ${1}$};
\node (iseven) at (-0.7,5.7) {\small $7$};
\node (iten) at (0.7,5.7) {\small $10$};
\node (two) at (-1,5) [place] {\small ${2}$};
\node (three) at (1,5) [place] {\small ${3}$};
\node (ifive) at (0.3,4.7) {\small $5$};
\node (inine) at (1.7,4.7) {\small $9$};
\node (four) at (0,4) [place] {\small ${4}$};
\node (ieight) at (0.6,3.7) {\small $8$};
\node (ten) at (2,4) [place] {\small ${10}$};
\node (ithree) at (2.7,3.7) {\small $3$};
\node (five) at (0.8,3) [place] {\small ${5}$};
\node (itwo) at (1.4,2.7) {\small $2$};
\node (six) at (1.6,2) [place] {\small ${6}$};
\node (isix) at (2.2,1.7) {\small $6$};
\node (eleven) at (3,3) [place] {\small ${11}$};
\node (eight) at (1.4,0) [place] {\small ${8}$};
\node (nine) at (3.2,0) [place] {\small ${9}$};
\node (seven) at (2.4,1) [place] {\small ${7}$};
\node (ione) at (1.7,0.7) {\small $1$};
\node (ifour) at (3,0.7) {\small $4$};

\draw[-,thick] (one) -- (two);
\draw[-,thick] (one) -- (three);
\draw[-,thick] (three) -- (four);
\draw[-,thick] (three) -- (ten);
\draw[-,thick] (four) -- (five) -- (six) -- (seven) -- (nine);
\draw[-,thick] (ten) -- (eleven);
\draw[-,thick] (seven) -- (eight);
\end{tikzpicture}
\hspace*{10 mm}
\begin{tikzpicture}[scale = 0.8]
\node (one) at (0,6) [place] {\small ${1}$};
\node (iseven) at (-0.7,5.7) {\small $7$};
\node (iten) at (0.7,5.7) {\small $10$};
\node (two) at (-1,5) [place] {\small ${2}$};
\node (three) at (1,5) [place] {\small ${3}$};
\node (ifive) at (0.3,4.7) {\small $5$};
\node (inine) at (1.7,4.7) {\small $9$};
\node (four) at (0,4) [place] {\small ${4}$};
\node (ieight) at (0.6,3.7) {\small $8$};
\node (five) at (2,4) [place] {\small ${5}$};
\node (itwo) at (2.7,3.7) {\small $2$};
\node (six) at (0.8,3) [place] {\small ${6}$};
\node (isix) at (1.4,2.7) {\small $6$};
\node (seven) at (3,3) [place] {\small ${7}$};
\node (ione) at (2.6,2.7) {\small $1$};
\node (ifour) at (3.7,2.7) {\small $4$};
\node (eight) at (1.6,2) [place] {\small ${8}$};
\node (nine) at (2.5,2) [place] {\small ${9}$};
\node (ten) at (4,2) [place] {\small ${10}$};
\node (ithree) at (4.7,1.7) {\small $3$};
\node (eleven) at (5,1) [place] {\small ${11}$};

\draw[-,thick] (one) -- (three) -- (five) -- (seven) -- (ten) -- (eleven);
\draw[-,thick] (one) -- (two);
\draw[-,thick] (three) -- (four) -- (six) -- (eight);
\draw[-,thick] (seven) -- (nine);
\draw[-,white] (0,-0.3) -- (1,-0.3); %%%This is just push the tree up.
\end{tikzpicture}
\end{center}
\caption{The depth-first search and breadth-first search
parking trees of the same parking function
$(7,5,10,7,3,6,1,4,3,1)$.}
\label{figure_parking_tree}
\end{figure}
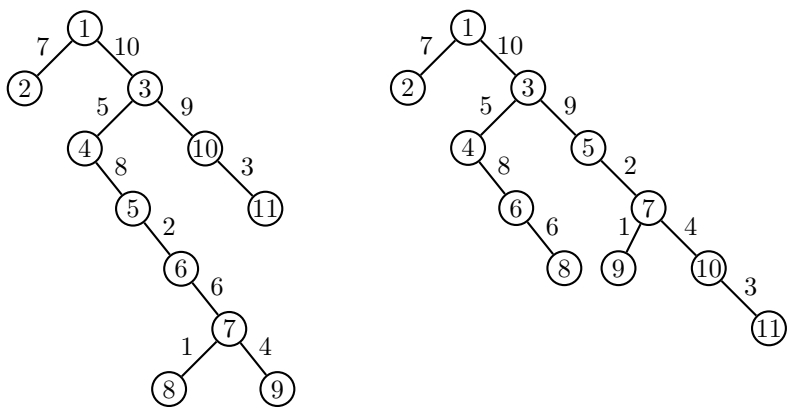

\begin{definition}
\label{definition_DFS_BFS}
Define the {\em depth-first search},
respectively {\em breadth-first search labeling},
of the vertices of a rooted plane tree as follows.
\begin{enumerate}
\item
The root of the tree is labeled $1$.
\item
Assume we have assigned the labels $1$ through $j$.
\begin{itemize}
\item[--]
For depth-first search, let $v$ be the vertex with the largest
label that has at least one unlabeled child.
\item[--]
For breadth-first search, let $v$ be the vertex with the smallest
label that has at least one unlabeled child.
\end{itemize}
Label the leftmost unlabeled child of $v$ by $j+1$.
\end{enumerate}
\end{definition}

A pair of parking trees encoding the same parking function is shown in
Figure~\ref{figure_parking_tree}. 
The tree on the left has the depth-first search labeling of the nodes,
whereas
the tree on the right has the breadth-first search labeling.

\begin{remark}
\label{remark_Yan}
{\em
Textbooks usually define the depth-first search and breadth-search
process for a connected graph in general, and use it to build a spanning tree
rooted at the starting point of the process. In such a
description of the breadth-first search process, the notion of {\em
level} is important: the starting vertex of the process has level $0$,
its neighbors have level $1$, and so on, all vertices that are added as
neighbors of a vertex of level $i$ have level $i+1$. We will not use
this notion of level, but it is worth noting that the breadth-first labeling process
we describe in Definition~\ref{definition_DFS_BFS} labels the vertices of a
plane tree level by level, where the level of a vertex is simply its
distance from the root. Finally, it should be emphasized that our
processes number the children of each vertex in the left-to-right order and
the resulting labeled tree can be drawn in exactly one way in the plane
such that this condition is satisfied. This aspect of our process is not
part of the general definition of a depth-first search or a
breadth-first search process as the graph considered may not even be
planar.
}           
\end{remark}

The trees shown in Figure~\ref{figure_parking_tree} are 
parking trees of the parking function whose Garsia--Haiman representation
is shown in Figure~\ref{figure_Garsia_Haiman}. If we read off the labels
on the parking tree in the breadth-first search order, we
obtain the same permutation $\pi=(7,10,5,9,8,2,6,1,4,3)$ that is
represented by the second Dyck path in
Figure~\ref{figure_Garsia_Haiman}.
\begin{definition}
Consider a parking tree on $n+1$ vertices. Define its {\em edge labeling
  permutation}~$\pi$ as the list of labels on the edges in the following
order:
\begin{enumerate}
\item For $i<j$ the edges connecting the vertex labeled $i$ to
  its children precede the labels of the edges connecting the vertex labeled $j$ to
  its children.
\item The labels of the edges connecting the same vertex to its children
  are listed in left-to-right (that is, increasing) order.
\end{enumerate}  
\end{definition}  

It is direct from the definitions that the edge labeling permutation of
a parking tree is identical to the permutation associated to the encoded
parking function. We only need to show that
Definition~\ref{definition_parking} always defines a parking function.  

This fact is part of the main result in this section: 
\begin{theorem}
\label{theorem_parkingtrees}  
Each parking tree on $n+1$ vertices encodes a parking function on $[n]$. 
Conversely, every parking function may be represented by a parking tree
in at least one way. Furthermore, the correspondence between parking
functions of $[n]$ and those parking trees on $n+1$ vertices that 
are labeled in the depth first (breadth-first) search order, is a bijection. 
\end{theorem}  
\begin{proof}
By definition, each parking tree on $n+1$ vertices encodes a function
$f:[n] \longrightarrow [n]$. (Note that $n+1$ is not in the range of $f$:
it labels a vertex with no children due to the fact that the labeling of
the vertices is increasing.) For each $i\leq n$, the set $f^{-1}([i])$
is the set of edge labels whose parent vertex is labeled with an element
of $[i]$. The corresponding edges form the edge set of a tree rooted at
the vertex labeled $1$: this is a consequence of the fact that the
labeling of the vertices is increasing. Not all vertices of this tree
are labeled with elements of the set $[i]$ because the vertex labeled
$i+1$ is the child vertex of an edge that belongs to this tree. Hence
the tree has at least $i+1$ vertices and at least $i$ edges. Therefore
$|f^{-1}([i])|\geq i$ holds for each $i\leq n$.

To prove the converse it suffices to show that there is a unique parking
tree whose vertices are labeled in the depth-first (breadth-first)
search order that encodes a given parking function. We prove the
statement for the depth-first search labeling. The breadth-first search
variant is completely analogous and left to the reader. Consider a
parking function $f:[n] \longrightarrow [n]$ corresponding to the pair
$(\pi,v)$ via the Garsia--Haiman bijection, described in
Subsection~\ref{subsection_parking_functions}. Here the associated
permutation $\pi$ is the ordered
list of the elements of $f^{-1}(1),f^{-1}(2),\ldots,f^{-1}(n)$, listed
in this order such that the elements belonging to the same fiber are
listed in increasing order, and $v$ is the Dyck word
$v=U^{q_{1}} D U^{q_{2}} D \cdots U^{q_{n}} D$ in which
$q_{i}=|f^{-1}(i)|$ for $i=1,2,\ldots,n$. A parking tree encodes this
parking function if and only if for each $i\in [n]$ the vertex labeled
$i$ satisfies the following two conditions:
\begin{enumerate}
\item The number of children of the vertex is  $q_{i}=|f^{-1}(i)|$.
\item The edge labeling permutation $\pi$ of the parking tree is the same
  as the permutation $\pi$ associated to $f$ by the Garsia--Haiman
  bijection.  
\end{enumerate}
Clearly, if a rooted plane tree whose vertices are labeled in an increasing
order satisfies Condition~(1) then there is only one way to
label the edges in such a way that the labeling satisfies
Condition~(2). Keeping in mind that balanced words $U^{q_{1}} D
U^{q_{2}} D \cdots U^{q_{n}} D$ of length $2n$ correspond bijectively to   
\L ukasziewicz words $f_{q_{1}}f_{q_{2}}\cdots f_{q_{n}} f_0$ of length
$n+1$, the theorem is a consequence of the well-known result stating
that for each \L ukasziewicz word $f_{q_{1}}f_{q_{2}}\cdots f_{q_{n}}
f_0$ of length $n+1$ there is a unique rooted plane tree whose vertex
labeled $i$ in the depth-first search order has $q_{i}$ children for
$i=1,2,\ldots,n$. A proof of this statement may be
found~\cite[Section~I.5.3]{Flajolet}. The analogous statement for the
breadth-first search order may be shown in a similar fashion,
see Proposition~\ref{proposition_BFS} in the Appendix.      
\end{proof}
  
Next we characterize the parking trees of $123$-avoiding parking
functions. 

\begin{proposition}
  A parking tree on $n+1$ vertices represents
  a $123$-avoiding parking function if and only if the following
  criteria are satisfied: 
\begin{enumerate}
\item Each vertex has at most two children.
\item The edge labeling permutation $\pi$ is $123$-avoiding.    
\end{enumerate}  
\end{proposition}  
The proof is a direct consequence of the definitions.
\begin{definition}
A {\em plane $0$-$1$-$2$ tree} is a rooted plane tree such that each
vertex has at most two children. A vertex of a rooted plane tree is a
{\em fork} if it has more than one child.
\end{definition}
Observe that a plane $0$-$1$-$2$ tree has one
more leaf than forks.
As noted in 
Lemma~\ref{lemma_j_UU} part~(b)
and in the proof of Theorem~\ref{theorem_Margie},
the entry $\gamma_{j} = C_{j} \cdot \binom{n}{2j}$
in the $\gamma$-vector of the associahedron
is the number of Motzkin paths of length $n$ with $j$ up steps. Next we
observe that this number is also the number of $0$-$1$-$2$ trees
with $j$ forks, by specializing the operation associating a \L
ukasziewicz word to each parking function. 
\begin{definition}
\label{definition_Motzkin_word}
Let $T$ be a plane $0$-$1$-$2$ tree with $n+1$ vertices 
labeled in some increasing order. Define the {\em associated
  Motzkin word $w_T$} as the word $x_{1}x_{2}\cdots x_{n}$ in which
$$
x_{i} =
\begin{cases}
  U  &\text{if the vertex labeled $i$ is a fork,}\\
  D  &\text{if the vertex labeled $i$ is a leaf,}\\
  H  &\text{otherwise.}\\
\end{cases}
$$  
\end{definition}  
\begin{proposition}
\label{proposition_012_Motzkin}
Let $T$ be a plane $0$-$1$-$2$ tree having $n+1$ vertices labeled in
an increasing order. Then the associated
Motzkin word $w_T$ encodes a Motzkin path of length $n$ where each
letter $U$ represents an up step, each letter $D$ represents a down step
and each letter $H$ represents a horizontal step.  
\end{proposition}  
\begin{proof}
Label the edges of $T$ with the elements of $[n]$ bijectively in such a
way that the labels on edges incident to the same parent vertex increase
left to right. This transforms $T$ into a parking tree encoding a parking
function $f:[n] \longrightarrow [n]$ such that for each $i\in [n]$ the value
$q_{i}=|f^{-1}(i)|$ is the number of children of the vertex labeled
$i$. Hence the associated \L ukasziewicz word $f_{q_{1}}f_{q_{2}}\cdots
f_{q_{n}} f_0$ satisfies $f_{q_{i}}\in \{0,1,2\}$ for each $i\in [n]$. After
deleting the last letter $f_0$, replacing each $f_{2}$ with $U$,
each~$f_{1}$ with $H$ and each $f_0$ with $D$,
we obtain the associated Motzkin word
that encodes a Motzkin path of length $n$.      
\end{proof}

In light of Theorem~\ref{theorem_Margie},
Proposition~\ref{proposition_gamma_contribution} may be rephrased as follows. 
\begin{definition}
\label{definition_type_B_tree}
Let $B$ be a sparse subset of $[n-1]$ and let $T$ be a plane $0$-$1$-$2$
tree on $n+1$ vertices labeled in an increasing order. Let us turn
$T$ into a parking tree by using the identity permutation of~$[n]$ as
the edge labeling permutation $\pi$. We say that
the vertex-labeled tree $T$ has {\em sibling type $B$} if the integers
$b$ and $b+1$ label edges sharing the same parent vertex if and only if
$b\in B$ holds.    
\end{definition}  
\begin{corollary}
\label{corollary_gamma_contribution}
Let $B$ be a sparse subset of $[n-1]$ and let $T$ be a plane $0$-$1$-$2$
tree on $n+1$ vertices, labeled in an increasing order. Then
the
coefficient of $x^{k}$ in $g_{n,|B|}(x)$ is the number of parking trees
obtained by assigning an edge labeling to $T$ in such a way that the
encoded parking function is $123$-avoiding and has exactly $k$ ascents.
\end{corollary}  
Corollary~\ref{corollary_gamma_contribution} motivates the question of
determining whether two vertex-labeled plane $0$-$1$-$2$ trees have the
same sibling type.
\begin{proposition}
Let $T$ and $T'$  be two plane $0$-$1$-$2$ trees on $n+1$ vertices
labeled in an increasing order. Then these vertex-labeled trees then have the
same number of vertices having the same sibling type if and only if, after
deleting all letters $D$ from their associated Motzkin words $w_{T}$
and $w_{T'}$, the two words become identical.  
\end{proposition}  
Indeed, as seen in the proof of
Proposition~\ref{proposition_012_Motzkin}, the Motzkin word $w_{T}$
carries the same information as the associated \L ukasziewicz word
$f_{q_{1}}f_{q_{2}}\cdots f_{q_{n}} f_0$ in which 
$q_{i}=0$ corresponds to the vertex labeled~$i$ having no children. As
we list the edges whose parent vertex is labeled $1$, $2$, and so on, and label
them consecutively with the elements of $[n]$ in increasing order, no
edge is labeled when we read a letter $D$ in the Motzkin word that
corresponds to an $f_{0}$ in the \L ukasziewicz word.

\section{The toric $g$-polynomial of the permutahedron}
\label{section_permutahedron}

The $n$-dimensional {\em permutahedron}
$\Perm_{n}$
is the convex hull of the $(n+1)!$
vectors $(\pi_{1}, \pi_{2}, \ldots, \pi_{n+1})$
in~$\Rrr^{n+1}$ where $(\pi_{1}, \pi_{2}, \ldots, \pi_{n+1})$ ranges over all
permutations in the symmetric group~$\SSSS_{n+1}$.
Note that this polytope lies in the hyperplane
$x_{1}+x_{2}+\cdots+x_{n+1} = \binom{n+2}{2}$.
An alternative definition is that the permutahedron
is the Minkowski sum of the line segments
$[e_{i},e_{j}]$ for $1 \leq i < j \leq n$.

It is well known that the $h$-polynomials of the permutahedra are the
Eulerian polynomials~\cite[Eq.~(4)]{Postnikov_Reiner_Williams}. For the
$n$-dimensional permutahedron we have 
\begin{align*}
h(\Perm_{n}, x)
& =
\sum_{\pi\in \SSSS_{n+1}} x^{\des(\pi)}.
\end{align*}
As pointed out in~\cite{Postnikov_Reiner_Williams},
the $\gamma$-vector of the
permutahedron was first described in terms of permutation statistics by
Shapiro, Getu and Woan~\cite{Shapiro_slides}.
In~\cite[Theorem~11.1]{Postnikov_Reiner_Williams}
Postnikov, Reiner and Williams
rephrased~\cite[Proposition~4]{Shapiro_slides} as follows.
\begin{theorem}[Shapiro--Getu--Woan]
The $\gamma$-polynomial of the $n$-dimensional permutahedron
$\Perm_{n}$ is given by
\begin{align*}
\gamma(\Perm_{n},x)
& =
\sum_{\pi \in \widehat{\SSSS}_{n+1}} x^{\des(\pi)} ,
\end{align*}
where $\widehat{\SSSS}_{n+1}$ denotes the set of permutations of
the set $[n+1]$ not containing a final descent or a double descent.  
\end{theorem}  
The technique of slides
used to prove their formulas was generalized to chordal nestohedra
in~\cite{Postnikov_Reiner_Williams}; see
Section~\ref{section_chordal_nestohedra}. As pointed out
in~\cite{Branden}, 
the ${\mathbb Z}_{2}^{n-1}$ action introduced in~\cite{Shapiro_slides} is
the same as Br\"and\'en's modified Foata--Strehl group
action discussed in Subsection~\ref{subsection_Foata_Strehl}.
Recall that this action {\em reverses} the labeling on the associated
Foata--Strehl trees, turning them into decreasing trees. On the other
hand, the actions of {\em hops} in~\cite{Postnikov_Reiner_Williams}
define precisely the restricted Foata--Strehl action. As noted in
Subsection~\ref{subsection_Foata_Strehl}, we may choose
the orbit representatives to be the right-adjusted Foata--Strehl
trees. By Corollary~\ref{corollary_right-adjusted} these are exactly the
Foata--Strehl trees representing permutations having no double descents
and no final descent. Using Definition~\ref{definition_right-adjusted} we    
may identify these with the labeled plane
increasing $0$-$1$-$2$ trees. This is how they are called
in~\cite{Bergeron_Flajolet_Salvy}. As stated
in~\cite{Bergeron_Flajolet_Salvy}, $\gamma_{j}(\Perm_{n})$
is the number of labeled plane increasing $0$-$1$-$2$ trees
on $n+1$ vertices having $j$ forks.

The $\gamma$-vector entries of the permutahedron are given as sequence
A101280 in OEIS~\cite{OEIS}.
The entries of the $\gamma$-vector of the $n$-dimensional permutahedron
satisfy the initial condition
$\gamma_{0}(\Perm_{n}) = 1$
and the following recurrence:
\begin{align}
\gamma_{j}(\Perm_{n})
& =
(j+1)\cdot \gamma_{j}(\Perm_{n-1}) + (2n+2-4j) \cdot \gamma_{j-1}(\Perm_{n-1}) ,
\end{align}
for $1 \leq j \leq \half{n}$.

A generating function formula for the $\gamma$-vector may be found using
the calculation in~\cite[Section~7]{Foata_Han}. Note that the numbers
$D_{n,k}$, whose 
generating function Foata and Han computed, count those orbits of the
Foata--Strehl group action where each tree has $k$ nonroot leaves. 
The modified Foata--Strehl group action has smaller orbits.  To obtain
the $\gamma$-vector entries we need to multiply each $D_{n,k}$ by the
appropriate power of $2$. Subject to these modifications one may derive the
following result; see sequence A101280 in OEIS~\cite{OEIS}.
\begin{proposition}[Peter Bala]
Introducing $r(x)=\sqrt{1-4x}$ and $w(x)=\frac{1-r(x)}{1+r(x)}$, we have
\begin{align*}
\sum_{n \geq 0} \gamma(\Perm_{n},x) \cdot \frac{t^{n+1}}{(n+1)!}
& =
\frac{1}{2x}\cdot \left(r(x)\cdot
\frac{1+w(x)\cdot e^{t\cdot r(x)}}{1-w(x)\cdot e^{t\cdot r(x)}}-1\right).
\end{align*}
\end{proposition}  

Keeping in mind Corollary~\ref{corollary_gamma_contribution}, we can
express the toric $g$-polynomial of the permutahedron in terms of an
ascent statistics on parking trees as follows. 

\begin{theorem}
The coefficient of $x^{k}$ in the toric $g$-polynomial
of the $n$-dimensional permutahedron~$\Perm_{n}$ is the number of parking trees
on $n+1$ nodes 
encoding $123$-avoiding parking function having exactly $k$ ascents.   
\label{theorem_Margie3}
\end{theorem}  
\begin{proof}
As noted above, $\gamma_{j}(\Perm_{n})$ is the number 
of (vertex-)labeled plane increasing $0$-$1$-$2$ trees $T$ on $n+1$ vertices
having $j$ forks~\cite{Bergeron_Flajolet_Salvy}. The tree $T$ has $j$ forks
if and only if its sibling type is $B$ for some $j$-element sparse
subset of $[n-1]$. By Corollary~\ref{corollary_gamma_contribution},
the coefficient of $x^{k}$ in $g_{n,|B|}(x)$ is the number of parking trees
obtained by adding an edge labeling to $T$ in such a way that the
encoded parking function is $123$-avoiding and has $k$ ascents.
The statement is now a direct consequence of
Theorem~\ref{theorem_peaks}.  
\end{proof}

The toric $g$-vectors of the permutahedra up to dimension $8$ are listed in
Table~\ref{table_gpol_perm}. 
\begin{table}[t]
\begin{tabular}{|c|| r| r| r| r| r|}  
\hline
  \backslashbox{$n$}{$j$} & 0 & 1 & 2 & 3& 4\\
\hline  
  1& 1 &     &      &      &     \\
  2& 1 &   3 &      &      &     \\
  3& 1 &   20 &      &     &      \\
  4& 1 &  115 &   40 &     &      \\
  5& 1 &  714 &  735 &     &       \\
  6& 1 & 5033&  10101 &  1225 &    \\  
  7& 1 & 40312 & 131068 & 45304&   \\
  8& 1 & 362871 & 1723716 & 1143996 & 67956  \\ 
\hline
\end{tabular}
\vspace*{1 mm}
\caption{The toric $g$-vector of the permutahedron up to dimension $8$.
See also~\cite[Page~195]{Stanley_toric_g}.}
\label{table_gpol_perm}  
\end{table}

The proof of Theorem~\ref{theorem_Margie3} may be adapted to show
several analogous results by focusing  on the following key ideas.  A
parking tree consists of an increasing vertex labeling and a compatible
edge labeling.
After fixing the vertex labeling one may add a compatible edge labeling.
Summing over all possible edge labelings results in a
contribution of $g_{n,j}(x)$ by
Corollary~\ref{corollary_gamma_contribution}.  If there is an expression
of $\gamma_{j}$ in the $\gamma$-vector of the polytope as a number of
vertex-labeled increasing plane $0$-$1$-$2$ trees then there is an
analogous result for the toric $g$-polynomials.

Postnikov, Reiner and Williams
showed~\cite[Section~10.2]{Postnikov_Reiner_Williams}
that the entry $\gamma_{j}$ in the
$\gamma$-vector of the $n$-dimensional associahedron is the number of
$312$-avoiding permutations in~$\SSSS_{n+1}$
having exactly $j$~descents, no double
descents and no final descent. They also note that a
permutation is $312$-avoiding if and only if the vertices of its
Foata--Strehl tree are labeled in the depth-first-search
order. Hence, adapting the proof of Theorem~\ref{theorem_Margie3}
yields the following result:
\begin{theorem}
The coefficient of~$x^{k}$ in the $\gamma$-polynomial
of the $n$-dimensional associahedron~$\Ass_{n}$ is the number of
parking trees whose vertices are labeled in the depth-first search
order that encode a $123$-avoiding parking function having
exactly $k$ ascents.
\end{theorem}
This statement is equivalent to
Theorem~\ref{theorem_Margie} since by
Theorem~\ref{theorem_parkingtrees} every parking function may be
uniquely encoded by a parking tree whose vertices are labeled in the
depth-first search order.

\section{Toric $g$-polynomials of chordal nestohedra}
\label{section_chordal_nestohedra}

In this section we review some of the terminology and results of
Postnikov, Reiner and Williams~\cite{Postnikov_Reiner_Williams}.  We then
outline how their expressions for the $\gamma$-vectors of chordal nestohedra can
be extended to results on the toric $g$-polynomials of these 
polytopes. Some cited facts and definitions also appear in earlier papers.
We refer the reader to~\cite{Postnikov_Reiner_Williams}
for a more complete bibliography.  

\begin{definition}
Given a finite set $S$,
a {\em building set} $\mathcal{B}$ on a set $S$ is a collection of
nonempty subsets of~$S$ such that
\begin{enumerate}
\item
if $I, J \in  \mathcal{B}$ and $I \cap J \neq \emptyset$
then $I \cup J \in \mathcal{B}$,
\item
$\mathcal{B}$ contains the singleton $\{i\}$ for all $i \in S$.
\end{enumerate}
The building set $\mathcal{B}$ on $S$ is {\em connected}
if $S \in \mathcal{B}$.
The connected components of a building set $\mathcal{B}$
are the maximal subsets under inclusion that are contained in $\mathcal{B}$.
For any subset $T \subseteq S$ we define the
{\em $T$-restricted building set}
$\mathcal{B}|_{T} = \{I \in \mathcal{B} : I \subseteq T\}$.
\end{definition}
For a graph $G$ on the vertex set $S$
with no loops and no multiple edges,
the graphical building set $\mathcal{B}(G)$
is the family of all nonempty
sets $I \subseteq S$ where 
the induced graph~$G|_{I}$ is connected.
Given a connected building set~$\mathcal{B}$,
we have the associated
{\em nestohedron} $P_{\mathcal{B}}$
defined as the Minkowski sum 
$$ P_{\mathcal{B}} =\sum_{I\in \mathcal{B}} \Delta_{I}, $$
where $\Delta_{I}$ is the convex hull of the basis vectors
$e_{i}$ satisfying $i\in I$.  The nestohedron~$P_{\mathcal{B}}$ is known
to be a simple polytope. A connected building set $\mathcal{B}$ is {\em
  chordal} if, for any of the sets 
$I = \{i_{1} < i_{2} < \cdots < i_{r}\}$ in~$\mathcal{B}$, all subsets
of the form $\{i_{s}, i_{s+1},\ldots, i_{r}\}$ also belong to $\mathcal{B}$. 
For chordal building sets the $h$-vector and the $\gamma$-vector of the
nestohedron~$P_{\mathcal{B}}$ may be expressed in terms of
descent statistics of {\em $\mathcal{B}$-permutations}, defined as
follows~\cite[Definition~8.7]{Postnikov_Reiner_Williams}.
\begin{definition}
The set of $\mathcal{B}$-permutations of a connected building on
the set $[n]$ is the set of all permutations $\pi$ of $[n]$ such that for any
$i \in [n]$, the two elements $\pi(i)$ and $\max\{\pi(1),\pi(2), \ldots, \pi(i)\}$
lie in the same connected component of the restricted building set
$B|_{\{\pi(1),\ldots,\pi(i)\}}$.
Denote the set of $\mathcal{B}$-permutations by
$\SSSS_{n}(\mathcal{B})$. 
\end{definition}
As stated in~\cite[Corollary~9.6]{Postnikov_Reiner_Williams},
the $h$-polynomial of the chordal nestohedra satisfy
\begin{equation}
\label{equation_hvector}  
h_{\mathcal{B}}(t) =\sum_{w\in\SSSS_{n}(\mathcal{B})} t^{\des(w)}.
\end{equation}
\begin{remark}
\label{remark_Btrees}
{\em More generally, Postnikov, Reiner and
Williams also have an expression for the $h$-vector of any nestohedron
associated to a  connected building
set~\cite[Corollary~8.4]{Postnikov_Reiner_Williams}. It
uses the descent statistics of {\em $\mathcal{B}$-trees} whose
definition we omit here
(see Section~8
of~\cite{Postnikov_Reiner_Williams}).
For chordal nestrohedra
the descent statistics of
$\mathcal{B}$-trees and $\mathcal{B}$-permutations 
coincide~\cite[Proposition~9.5]{Postnikov_Reiner_Williams}. 
Note that the {\em $\mathcal{B}$-trees} are
unrelated to the trees we consider in this section. 
}   
\end{remark}

For our purposes, the most important result of  Postnikov,
Reiner and Williams 
is as follows~\cite[Theorem~11.6]{Postnikov_Reiner_Williams}.
\begin{theorem}[Postnikov--Reiner--Williams]
The $\gamma$-polynomial of
a chordal nestohedron $P_{\mathcal{B}}$ is given by
\begin{equation}
\label{equation_gamma}    
\gamma_{\mathcal{B}}(t) =
\sum_{w\in\widehat{\SSSS}_{n}(\mathcal{B})} t^{\des(w)},
\end{equation}
where $\widehat{\SSSS}_{n}(\mathcal{B})$ is the set of all
$\mathcal{B}$-permutations containing no double descents and no final
descents.
\end{theorem}
The proof of this theorem
relies on considering a
${\mathbb Z}_{2}^{n-1}$-action of {\em $\mathcal{B}$-hops} that is similar
to the restricted Foata--Strehl action.
In case of the permutahedron,
the $\mathcal{B}$-hops coincide with
the restricted Foata--Strehl action.
For the definition and detailed study of the
$\mathcal{B}$-hops, we refer the reader
to~\cite[Section~11]{Postnikov_Reiner_Williams}. For us it suffices to
observe the following direct consequence of
equation~\eqref{equation_gamma}.
\begin{corollary}
\label{corollary_gamma}
In the $\gamma$-vector of a chordal nestohedron $P_{\mathcal{B}}$ the entry
$\gamma_{j}$ is the number of right-adjusted Foata-Strehl trees
representing a $\mathcal{B}$-permutation and having $j$ forks. 
\end{corollary}  
Using Corollary~\ref{corollary_gamma} we obtain the following generalization of
Theorem~\ref{theorem_Margie3}. 
\begin{theorem}
Let $\mathcal{B}$ be a connected chordal building set on $[n+1]$.
The coefficient of $x^{k}$ in the toric $g$-polynomial
of the nestohedron $P_{\mathcal{B}}$ is the number of parking trees
$T$ satisfying the following two conditions:
\begin{enumerate}
\item $T$ encodes a $123$-avoiding parking function $f:[n] \longrightarrow
  [n]$  having exactly $k$ ascents.
\item Removing the labeling on the edges of $T$ results in the
  right-adjusted Foata--Strehl tree of a $\mathcal{B}$-permutation
  having no double descents and no final descent.  
\end{enumerate}
In the above statements transform each $0$-$1$-$2$ tree
into a right-adjusted
Foata--Strehl tree by declaring the only child of a parent
to be a right child. 
\label{theorem_Margie4}
\end{theorem}
The proof is a straightforward adaptation of the proof of
Theorem~\ref{theorem_Margie3} and is therefore omitted.

\begin{example}
{\em The $n$-dimensional cube is combinatorially equivalent to the
{\em Stanley--Pitman polytope} of the same dimension. As shown
in~\cite[Subsection~10.5]{Postnikov_Reiner_Williams}, the 
Stanley--Pitman polytope is a chordal nestohedron $P_{\mathcal{B}}$,
where the building set is given by
\begin{align*}
\mathcal{B}
& =
\{\{i\}, [i,n+1] : 1 \leq i \leq n+1\} .
\end{align*}
The $\mathcal{B}$-permutations are exactly the permutations
$\pi \in \SSSS_{n+1}$ satisfying
$\pi(1) < \pi(2) < \cdots < \pi(k) > \pi(k+1) > \cdots > \pi(n+1)$ for
some $k \in [n+1]$. The only $\mathcal{B}$-permutation with no double
descent and no final descent is the identity permutation.
Its Foata-Strehl tree is the right-adjusted path
$1$--$2$--$\cdots$--$(n+1)$. Labeling  the edges amounts to selecting a
permutation $\tau$ of $[n]$ where 
$\tau(i)$ is the label of the edge $(i,i+1)$. The encoded parking
function is $123$-avoiding if and only if the permutation $\tau$ is
$123$-avoiding. Hence we recover
Corollary~\ref{corollary_g_of_cube}.
Keeping Remark~\ref{remark_Btrees} in mind, notice that all Foata--Strehl
trees associated to $\mathcal{B}$-permutations are paths, whereas
$\mathcal{B}$-trees $T_I$ defined
in~\cite[Subsection~10.5]{Postnikov_Reiner_Williams} have definitely
more than one path when a proper subset $I$ of $[n]$ is considered.
}   
\end{example}  

\begin{example}
{\em The associahedron is discussed
in~\cite[Subsection~10.5]{Postnikov_Reiner_Williams}, where it is
represented as a chordal nestohedron $P_{\mathcal{B}}$
with the building set given by
\begin{align*}
\mathcal{B}
& =
\{[i,j] : 1 \leq i \leq j \leq n+1\}.
\end{align*}  
In this case
$\mathcal{B}$-permutations are exactly the $312$-avoiding
permutations. It is easy to check, and left to the reader, that a
permutation is $312$-avoiding if and only if the vertices of the
corresponding Foata--Strehl tree are labeled in the depth-first search
order. In the proof of Theorem~\ref{theorem_parkingtrees} we established
that every parking function has a unique parking tree representation
such that the vertices are labeled in the depth-first search
order. Parking trees of $123$-avoiding parking functions are plane
$0$--$1$--$2$ trees which we may identify with right-adjusted
Foata--Strehl trees. Hence we recover Theorem~\ref{theorem_Margie}.   
Keeping Remark~\ref{remark_Btrees} in mind, 
in this particular case
$\mathcal{B}$-trees coincide with Foata--Strehl trees
of $\mathcal{B}$-permutations.
}   
\end{example}

\begin{example}
{\em The permutahedron is discussed
in~\cite[Subsection~10.1]{Postnikov_Reiner_Williams}.
Here the building set is
\begin{align*}
\mathcal{B}
& =
2^{[n+1]} - \{\emptyset\}
=
\{I \subseteq [n+1]: I \neq \emptyset\} .
\end{align*}
Here every permutation is a $\mathcal{B}$-permutation and
Theorem~\ref{theorem_Margie4} specializes to
Theorem~\ref{theorem_Margie3}.
Keeping Remark~\ref{remark_Btrees} in mind,
the $\mathcal{B}$-trees here are linear orders (labeled paths),
whereas all Foata--Strehl trees appear in this example.       
}   
\end{example} 

\begin{question}
{\em
Consider the graph on the vertex set $[n+1]$
where $(i,j)$ is an edge if $\max(i,j) \geq r+1$.
The edge set of this graph is $E(K_{n+1}) - E(K_{r})$
where the subgraph $K_{r}$ is on the vertex set~$[r]$.
The associated building set is chordal and is given by
\begin{align*}
{\mathcal B}
& =
\{\{i\} : i \in [r]\}
\cup
\{I \subseteq [n+1] : \max(I) \geq r+1\} .
\end{align*}
The associated nestohedra interpolate between the stellohedron
when $r=n$
(see~\cite[Subsection~3.5.4]{Aisbett},
\cite[Subsection~10.4]{Postnikov_Reiner_Williams}
and~\cite[Section~5]{Shareshian_Wachs})
and the permutahedron when $r=1$.
A permutation~$\pi$ is a $\mathcal{B}$-permutation
if its longest initial segment of elements less than or equal to $r$
is increasing.
By removing this initial segment, we have
a bijection between the set of $\mathcal{B}$-permutations
and {\em partial permutations}~$\tau$ of $[n+1]$
where the ``missing elements'' of $\tau$
are from the set~$[r]$.
For example, for $n=8$, $r=6$ and $\pi=(1,5,7,4,9,2,6,3,8)$ we
have $\tau=(7,4,9,2,6,3,8)$. 
Comparing the Foata--Strehl tree of $\pi$ with
the Foata--Strehl tree of $\tau$, the latter one is obtained from the
former by removing the right-adjusted path $1$--$5$--$7$.
Are there explicit expressions for the $\gamma$-
and the toric $g$-vectors of the associated nestohedron?
}   
\end{question}

\section{Concluding remarks}
\label{section_Concluding_remarks}

By a result of Billera, Chan and Liu~\cite{Billera_Chan_Liu} (see also
Corollary~\ref{corollary_nccomplex}) the coefficients in the toric
$g$-polynomial of the cube are face numbers of a simplicial complex. Can
this result be generalized to other simple polytopes? As stated in
Conjecture~\ref{conjecture_simplicial}, we suspect that
the answer is ``yes'':  perhaps one can combine the proof of
Theorem~\ref{theorem_Margie3} and the constructions of Nevo and
Petersen~\cite{Nevo_Petersen} who proved that the $\gamma$-vectors of
several simple polytopes are $f$-vectors of simplicial complexes.

Besides Conjecture~\ref{conjecture_simplicial}
that we suspect to hold for structural reasons,
strong numerical evidence supports the following two conjectures.
\begin{conjecture}
The toric $g$-contribution polynomials~$g_{n,j}(x)$
for $0 \leq j \leq \lfloor n/2 \rfloor$
are real-rooted.
\end{conjecture}

\begin{conjecture}
The $g$-polynomials of
the $n$-dimensional cyclohedron
and chordal nestohedra are real-rooted.
\end{conjecture}

Some comments concerning the occurrence of parking trees in the
literature are in order. 

To obtain the labeled trees introduced in~\cite[Section 13.2.3]{Yan_chapter}
from our representations, one needs to take the mirror images of our
illustrations, move each of our edge labels to the lower end of the
labeled edge (this idea of label shifting also appears
in~\cite{Irving_Rattan}), and label the root with~$0$.   Our labeling of
the vertices is 
called a {\em vertex ordering} in~\cite[Section
  13.2.3]{Yan_chapter}. Yan defines these orderings as the output of a
recursive process which relies on setting a {\em choice function}
first.  It is not hard to verify that the resulting vertex labelings all
have the increasing property stated in
Definition~\ref{definition_increasing_vertex_labeling}. Note that the class of
parking trees we consider in Sections~\ref{section_permutahedron} and
\ref{section_chordal_nestohedra} is broader: we need to use {\em all}
increasing vertex labelings, not only the ones that arise as the output 
of some recursive process. 

Generalizations of~\cite[Example~13.13]{Yan_chapter},
respectively~\cite[Example~13.14]{Yan_chapter}, to $d$-parking trees
may be found in~\cite{Ehrenborg_Hetyei} and in~\cite[Section~6]{Hetyei_genlin},
respectively. 
A key feature of all parking trees discussed in~\cite[Section
  13.2.3]{Yan_chapter} is that $a_i=a_j$ holds in the parking function
exactly when $i$ and $j$ label siblings in the parking tree. This
property is {\em not present} in the parking trees of
Kreweras~\cite{Kreweras_suites} and
Knuth~\cite[Section~13.2.2]{Yan_chapter}, nor in the $d$-parking trees
of Irving and Rattan~\cite{Irving_Rattan}.

\section*{Acknowledgments}

The second author thanks the Department of Mathematics at
the University of Kentucky for research visits.
The authors thank Andrew Howroyd for suggesting
Corollary~\ref{corollary_associahedron_2m_m}.
This work was partially supported by grants from the
Simons Foundation
(\#429370 to Richard~Ehrenborg,
\#245153 and \#514648 to G\'abor~Hetyei), and by the NSF grant
DMS-2247382 to Margaret~Readdy. 

\appendix

\section{\L ukasziewicz words}

\begin{proposition}
\label{proposition_BFS}
Given a \L ukasziewicz word $f_{q_{1}}f_{q_{2}}\cdots f_{q_{n}} f_0$ of length
$n+1$, there is a unique rooted plane tree on $n+1$ vertices whose
vertices labeled in the breadth-first search order have the property
that the number of children of the vertex labeled $i$ is $q_{i}$. 
\end{proposition}  
\begin{proof}
We show the statement by analyzing the breadth-first search labeling
process.
For each $0 \leq i \leq n$ let $m(i)$ be the number
of vertices labeled after visiting the first $i$ vertices and labeling
their children. At the beginning of the process only the root vertex is
labeled $1$, but not yet visited, hence we set $m(0)=1$ for
the empty word $\varepsilon$.
The quantity $m(i)$ is given by
adding the number of children of each visited vertex, that is, we have
\begin{equation}
\label{equation_labeled}  
m(i)=1+w(f_{q_{1}})+\cdots+w(f_{q_{i}})
\quad\text{where $w(f_{q})=q-1$.} 
\end{equation}
After visiting the first $i$ vertices, the number of vertices already
labeled but not yet visited is the difference $m(i)-i$.
The labeling process can continue only if this number
is positive at each stage of the process.
Furthermore, at the end of the process we must have
$m(n)=n+1$, that is, $m(n)-n=1$.
Conversely, if the numbers $m(i)$
satisfy the above mentioned conditions, we can use the breadth-first
search labeling process to reconstruct our rooted plane tree in a
unique fashion.
Let us define the {\em level} $\ell(f_{q_{1}}\cdots f_{q_{i}})$
of the word $f_{q_{1}}\cdots f_{q_{i}}$ by setting
\begin{align*}
\ell(f_{q_{1}}\cdots f_{q_{i}}) & =m(i)-i-1.
\end{align*}
By~\eqref{equation_labeled} and the above reasoning, the word
$f_{q_{1}}\cdots f_{q_{n}}f_{0}$ encodes a rooted plane tree if and only if
$\ell(f_{q_{1}}\cdots f_{q_{i}})\geq 0$ holds for $i\in [n]$ and  we have
$\ell(f_{q_{1}}\cdots f_{q_{n}}f_{0})=0$. This is precisely the condition requiring
that the word is a \L ukasziewicz word of length $n+1$.   
\end{proof}

\newcommand{\journal}[6]{{\sc #1,} #2, {\it #3} {\bf #4} (#5), #6.}
\newcommand{\book}[4]{{\sc #1,} ``#2,'' #3, #4.}
\newcommand{\bookf}[5]{{\sc #1,} ``#2,'' #3, #4, #5.}
\newcommand{\arxiv}[3]{{\sc #1,} #2, {\tt #3}.}
\newcommand{\preprint}[3]{{\sc #1,} #2, preprint {(#3)}.}
\newcommand{\oeis}[3]{{\sc #1,} #2, {#3}.}
\newcommand{\preparation}[2]{{\sc #1,} #2, in preparation.}
\newcommand{\dissertation}[4]{{\sc #1,} ``#2,'' #3, #4.}

\end{document}